\documentclass[12pt]{article}

\usepackage{verbatim}

\usepackage{amsfonts}

\usepackage{amsthm}

\usepackage{amssymb}

\usepackage{amsmath}

\usepackage{mathabx}

\usepackage{enumerate}

\usepackage[all]{xy}

\usepackage{graphicx}

\usepackage{hyperref}

\newcommand{\N}{\mathbb{N}}

\newcommand{\R}{\mathbb{R}}

\newcommand{\C}{\mathbb{C}}

\newcommand{\K}{\mathcal{K}}

\newcommand{\Hawaii}{Hawai\kern.05em`\kern.05em\relax i}

\newcommand{\ine}{\in_\epsilon}

\newcommand{\ind}{\in_\delta}

\newcommand{\bt}{\beta^{-1}}

\theoremstyle{plain}
\newtheorem{theorem}{Theorem}[section]
\newtheorem{lemma}[theorem]{Lemma}
\newtheorem{corollary}[theorem]{Corollary}
\newtheorem{proposition}[theorem]{Proposition}

\newtheorem{definition-theorem}[theorem]{Definition / Theorem}

\newtheorem*{conjecture*}{Conjecture}
\newtheorem*{theorem*}{Theorem}

\theoremstyle{definition}
\newtheorem{definition}[theorem]{Definition}
\newtheorem{example}[theorem]{Example}

\theoremstyle{remark}
\newtheorem{remark}[theorem]{Remark}

\newtheorem*{example*}{Example}  
\newtheorem*{remark*}{Remark}

\begin{document}
\title{Approximate ideal structures and K-theory}
\author{Rufus Willett}

\maketitle

\begin{abstract}
We introduce a notion of approximate ideal structure for a $C^*$-algebra, and use it as a tool to study $K$-theory groups.  The notion is motivated by the classical Mayer-Vietoris sequence, by the theory of nuclear dimension as introduced by Winter and Zacharias, and by the theory of dynamical complexity introduced by Guentner, Yu, and the author.   A major inspiration for our methods comes from recent work of Oyono-Oyono and Yu in the setting of controlled $K$-theory of filtered C*-algebras; we do not, however, use that language in this paper. 

We give two main applications.  The first is a vanishing result for $K$-theory that is relevant to the Baum-Connes conjecture.  The second is a permanence result for the K\"{u}nneth formula in $C^*$-algebra $K$-theory: roughly, this says that if $A$ can be decomposed into a pair of subalgebras $(C,D)$ such that $C$, $D$, and $C\cap D$ all satisfy the K\"{u}nneth formula, then $A$ itself satisfies the K\"{u}nneth formula.  

\end{abstract}

\tableofcontents

\section{Introduction}

\subsection*{Approximate ideals structures and long exact sequences}

Let $C$ and $D$ be $C^*$-subalgebras of a $C^*$-algebra $A$.  There is a natural sequence of maps
\begin{equation}\label{general les}
K_1(C\cap D) \stackrel{\iota}{\to} K_1(C)\oplus K_1(D) \stackrel{\sigma}{\to} K_1(A) \stackrel{\partial}{\dashrightarrow} K_0(C\cap D) \stackrel{\iota}{\to} K_0(C)\oplus K_0(D) 
\end{equation}
of $K$-theory groups where the solid arrows labeled $\iota$ and $\sigma$ are defined respectively by $\iota(\kappa):=(\kappa,-\kappa)$ and $\sigma(\kappa,\lambda):=\kappa+\lambda$.
The dashed arrow labeled $\partial$ does not exist in general, but in the very special case that $C$ and $D$ are ideals in $A$ such that $A=C+D$, one can canonically fill it in.  Indeed, the dashed arrow is then a boundary map in a six-term exact sequence
$$
\xymatrix{ K_1(C\cap D) \ar[r]^-\iota & K_0(C)\oplus K_0(D) \ar[r]^-\sigma & K_0(A) \ar[d]^-\partial \\
K_1(A) \ar[u]^-\partial & K_1(C)\oplus K_1(D) \ar[l]^-\sigma & K_1(C\cap D) \ar[l]^-\iota }.
$$
This is the $C^*$-algebraic analogue of the classical Mayer-Vietoris sequence associated to a cover of a topological space by two open sets.

The main technical tools developed in this paper are partial exactness results for the sequence in line \eqref{general les} that hold under less rigid assumptions than $C$ and $D$ being ideals.  These tools have interesting consequences even for many simple $C^*$-algebras, where there are no non-trivial ideals.  Looking at the diagram in line \eqref{general les} in more detail,
\begin{equation}\label{pre mv}
K_1(C\cap D) \stackrel{\iota}{\to} \underbrace{K_1(C)\oplus K_1(D)}_{(III)} \stackrel{\sigma}{\to} \underbrace{K_1(A)}_{(II)} \stackrel{\partial}{\dashrightarrow} \underbrace{K_0(C\cap D)}_{(I)} \stackrel{\iota}{\to} K_0(C)\oplus K_0(D) 
\end{equation}
we establish partial exactness results at each of the three places marked (I), (II), and (III), under progressively more stringent assumptions.  Exactness at point (I) is the easiest to prove, and is automatic: if $\iota(\kappa)=0$ for some $\kappa\in K_0(C\cap D)$, one can always canonically construct a class in $K_1(A)$ that is the `reason' for its being zero in some sense.

For exactness in the positions marked (II) and (III) in line \eqref{pre mv}, we need more assumptions.  Here are the technical definitions.

\begin{definition}\label{intro decomp}
Let $A$ be a $C^*$-algebra, and let $\mathcal{C}$ be a set of pairs $(C,D)$ of $C^*$-subalgebras of $A$.  Then $A$ admits an \emph{approximate ideal structure} over $\mathcal{C}$ if for any $\delta>0$ and any finite subset $\mathcal{F}$ of $A$ there exists a positive contraction $h$ in the multiplier algebra of $A$ and a pair $(C,D)\in \mathcal{C}$ such that:
\begin{enumerate}[(i)]
\item \label{dec com} $\|[h,a]\|< \delta$ for all $a\in \mathcal{F}$;
\item \label{dec split} $d(ha,C)<\delta$ and $d((1-h)a,D)< \delta $ for all $a\in \mathcal{F}$;
\item \label{dec split 2} $d((1-h)ha,C\cap D)<\delta$ and $d((1-h)h^2a,C\cap D)<\delta$ for all $a\in \mathcal{F}$.
\end{enumerate}
\end{definition} 
The pair $\{h,1-h\}$ should be thought of as a `partition of unity' on $A$, splitting it into two `parts' $C$ and $D$ that are simpler than the original.  We discuss examples below, but keep the discussion on an abstract level for now.

These conditions allow us to prove a version of exactness at position (II) in line \eqref{pre mv}: roughly this says that if $A$ admits an approximate ideal structure over $\mathcal{C}$, then for any class $[u]$ in $K_1(A)$ one can find a pair $(C,D)\in \mathcal{C}$ and build a class $\partial(u)\in K_0(C\cap D)$ such that if $\partial(u)=0$, then $[u]$ is in the image of $\sigma$.  

The first of our main results is as follows. 

\begin{theorem}\label{pre main}
Say that $A$ admits an approximate ideal structure over a set $\mathcal{C}$ such that for all $(C,D)\in \mathcal{C}$, the $C*$-algebras $C$, $D$, and $C\cap D$ have trivial $K$-theory.  Then $A$ has trivial $K$-theory.  
\end{theorem}
This result is already quite powerful: for example, it allows one to reprove the main theorem on the Baum-Connes conjecture of Guentner, Yu, and the author from \cite{Guentner:2014bh} without the need for the controlled $K$-theory methods used there.  

In order to get our results on the K\"{u}nneth formula, we need an exactness property at position (III) in line \eqref{pre mv}; unfortunately, this needs the stronger assumption on $A$ defined below.  

\begin{definition}\label{intro exc}
Let $A$ be a $C^*$-algebra and $\mathcal{C}$ a set of pairs $(C,D)$ of $C^*$-subalgebras of $A$.  Then $A$ admits a \emph{uniform approximate ideal structure over $\mathcal{C}$} if it admits an approximate ideal structure over $\mathcal{C}$, and if in addition the following property holds.  For all $\epsilon>0$ there exists $\delta>0$ such that for any $C^*$-algebra $B$, if $c\in C\otimes B$ and $d\in D\otimes B$ satisfy $\|c-d\|<\delta$, then there exists $x\in (C\cap D)\otimes B$ with $\|x-c\|<\epsilon$ and $\|x-d\|<\epsilon$.  
\end{definition}

The above definition is satisfied, for example, if all the pairs $(C,D)\in \mathcal{C}$ are pairs of ideals.  However, this is too much to ask if one wants  applications that go beyond well-understood cases.  There are non-trivial examples, but we will discuss these until later.

Here is our second main theorem.

\begin{theorem}\label{main}
Let $A$ be a $C^*$-algebra.  Assume that $A$ admits a uniform approximate ideal structure over $\mathcal{C}$, and that for each $(C,D)\in \mathcal{C}$, $C$, $D$, and $C\cap D$ satisfy the K\"{u}nneth formula.  Then $A$ satisfies the K\"{u}nneth formula.
\end{theorem}

Before moving on to examples, let us digress slightly to give background on the K\"{u}nneth formula for readers unfamiliar with this.

\subsection*{The K\"{u}nneth formula}

One of the main results in this paper is about the K\"{u}nneth formula, which concerns the external product map
$$
\times:K_*(A\otimes B)\to K_*(A)\otimes K_*(B)
$$
in $C^*$-algebra $K$-theory.  This product is as a special case of the very general Kasparov product, but can also be defined in an elementary way: see for example \cite[Section 4.7]{Higson:2000bs}.   A $C^*$-algebra $A$ is said to satisfy the \emph{K\"{u}nneth formula} if for any $C^*$-algebra $B$ with free abelian $K$-groups, the product map above is an isomorphism.  

Study of the K\"{u}nneth formula seems to have been initiated by Atiyah \cite{Atiyah:1962aa} in the commutative case, and in general by Schochet \cite{Schochet:1982aa}.  In particular, these authors showed (in the relevant contexts) that $A$ satisfies the K\"{u}nneth formula in the above sense if and only if for any $B$ there is a canonical short exact sequence
$$
0 \to K_*(A)\otimes K_*(B) \stackrel{\times}{\to} K_*(A\otimes B)\to \text{Tor}(K_*(A),K_*(B)) \to 0.
$$
This short exact sequence is a useful computational tool, so it is desirable to know for which $C^*$-algebras the K\"{u}nneth formula holds.  One can see the K\"{u}nneth formula as a sort of `dual form' of the universal coefficient theorem (UCT).  Thus another motivation for studying the K\"{u}nneth formula is as it forms a simpler proxy for the UCT.

The class of $C^*$-algebras known to satisfy the K\"{u}nneth formula is large.  Atiyah \cite{Atiyah:1962aa} essentially showed that commutative $C^*$-algebras satisfy the K\"{u}nneth formula.  It follows that any $C^*$-algebra\footnote{For this and the next paragraph, all $C^*$-algebras are separable.} that is $KK$-equivalent to a commutative $C^*$-algebra satisfies the K\"{u}nneth formula.  The class of such $C^*$-algebras is exactly the class satisfying the UCT\footnote{This is implicit in the original work of Rosenberg and Schochet \cite{Rosenberg:1987bh}, and was made explicit by Skandalis in \cite[Proposition 5.3]{Skandalis:1988rr}.}.  Hence the UCT is stronger the K\"{u}nneth formula.  

The UCT is in fact \emph{strictly} stronger that the K\"{u}nneth formula: this follows from  combining work of Chabert, Echterhoff, and Oyono-Oyono \cite{Chabert:2004fj}, of Lafforgue \cite{Lafforgue:2009ss}, and of Skandalis \cite{Skandalis:1988rr}.  Indeed, it follows from the `going down functor' machinery of \cite{Chabert:2004fj} that if $G$ is any group that satisfies the Baum-Connes conjecture with coefficients, then $C^*_r(G)$ satisfies the K\"{u}nneth formula.  Thanks to \cite{Lafforgue:2009ss}, this applies in particular when $G$ is a hyperbolic group.  On the other hand, results of \cite{Skandalis:1988rr} imply\footnote{The result as stated here is not exactly in Skandalis's paper \cite{Skandalis:1988rr}, but it follows from Skandalis's ideas, plus more recent advances in geometric group theory: see \cite[Theorem 6.2.1]{Higson:2004la} for a discussion of the version stated.} that if $G$ is an infinite, hyperbolic, property (T) group, then $C^*_r(G)$ does not satisfy the UCT.  

Other results extending the range of validity of the K\"{u}nneth formula include work of B\"{o}nicke and Dell'Aiera \cite{Bonicke:2018aa}, which extends the results of \cite{Chabert:2004fj} from groups to groupoids; and work of Oyono-Oyono and Yu \cite{Oyono-Oyono:2016qd} which uses the methods of controlled $K$-theory developed by those authors \cite{Oyono-Oyono:2011fk}, and based on older ideas of Yu \cite{Yu:1998wj}.  The work of Oyono-Oyono and Yu was the main technical inspiration for this paper, and we say more on this below.

Despite all these positive results, there are known to be $C^*$-algebras that do not satisfy the K\"{u}nneth formula.  The only way we know to produce such examples is based on the existence of non $K$-exact $C^*$-algebras: see the discussion in \cite[Remark 4.3 (1)]{Chabert:2004fj}.  We do not know of an exact $C^*$-algebra that does not satisfy the K\"{u}nneth formula.

\subsection*{Examples}

Our definitions were motivated partly by the theory of nuclear dimension.  Indeed, we can weaken Definition \ref{intro decomp} as follows.

\begin{definition}\label{weak ais}
A $C^*$-algebra $A$ admits a \emph{weak approximate ideal structure} over $\mathcal{C}$ if the conditions from Definition \ref{intro decomp} are satisfied, with condition \eqref{dec split 2} on intersections omitted. 
\end{definition}
\noindent In Appendix \ref{nd1 app}, we show\footnote{This result was pointed out to us by Wilhelm Winter.} that if $A$ is a (separable) $C^*$-algebra of nuclear dimension one, then $A$ admits a weak approximate ideal structure over a class of pairs of subhomogeneous $C^*$-subalgebras with very simple structure.  This result is not enough to deduce $K$-theoretic consequences with our current techniques; nonetheless, it provides evidence that our conditions are natural from the point of view of general $C^*$-algebra structure theory.

In Appendix \ref{fdc app}, we discuss examples coming from groupoids.  In joint work with Guentner and Yu \cite[Appendix A]{Guentner:2014bh}, we introduced a notion of a \emph{decomposition} of an \'{e}tale groupoid.  In Appendix \ref{fdc app}, we show that such decompositions naturally give rise to approximate ideal structures of the associated reduced groupoid $C^*$-algebras, and moreover that we get uniform approximate ideal structures in this way if the groupoids involved are ample.  We use this to show that a large class of reduced groupoid $C^*$-algebras satisfy the K\"{u}nneth formula\footnote{Similar results have been proved recently (and earlier than the current work) by Oyono-Oyono using the methods of controlled $K$-theory.}.

\subsection*{Inspiration and motivation}

This paper was inspired by the work of Oyono-Oyono and Yu in \cite{Oyono-Oyono:2016qd} on the K\"{u}nneth formula in controlled $K$-theory.  It owes a great deal to their work, both conceptually and in some technical details: in particular, the key idea to use a sort of approximate Mayer-Vietoris sequence comes directly from \cite{Oyono-Oyono:2016qd}, and the difficult proof of Proposition \ref{sigma lem} is based closely on their work. 
A major difference of our work from \cite{Oyono-Oyono:2016qd} in that we do not use controlled $K$-theory, only usual $K$-theory groups.  We do not use filtrations on our $C^*$-algebras, and we do not need (nor do we get results on) a `controlled' version of the K\"{u}nneth formula.  It is not clear to us what the difference is between the range of validity of our results and those of \cite{Oyono-Oyono:2016qd}; we suspect that there is a large overlap.

We were motivated largely by the theory of nuclear dimension \cite{Winter:2010eb}: we wanted to narrow the gap between the sort of structural results that one can use to deduce $K$-theoretic consequences, and the sort of structural results that are known for $C^*$-algebras of finite nuclear dimension.

\subsection*{Outline of the paper}

Section \ref{bound sec} introduces a general notion of `boundary classes', and shows that such classes have good properties with respect to the sequence of maps in line \eqref{pre mv}: roughly, we prove a weak form of exactness at position (II) in line \eqref{pre mv}.  The discussion in Section \ref{bound sec} does not give a construction of boundary classes: this is done in Section \ref{decomp bound sec} using approximate ideal structures.  We then prove Theorem \ref{pre main}, our first main goal of the paper.

In Section \ref{mo bound sec}, we prove exactness at position (I) in line \eqref{pre mv}; this is simpler than exactness at position (II), but is postponed until later as it is not needed for the proof of Theorem \ref{pre main}.  We also collect together some other technical results on the boundary map that are needed later. Exactness at position (III) in line \eqref{pre mv} is handled in Section \ref{sum sec}: this is the most difficult of our exactness properties, both to prove and to use.

Section \ref{prod sec} recalls some facts about the product in $K$-theory, and proves that the products maps interact well with our boundary classes.  Section \ref{bott sec} recalls material about the inverse Bott map that we need for the technical proofs.  We prove Theorem \ref{main} in Sections \ref{surj sec} and \ref{inj sec}, which handle the surjectivity and injectivity halves respectively.   

Finally, there are two appendices that discuss examples.  The first of these, Appendix \ref{nd1 app} shows that $C^*$-algebras of nuclear dimension one have weak approximate ideal structures.  Appendix \ref{fdc app} gives examples of (uniformly) approximate ideal structures coming from groupoid theory, and briefly discusses consequences for the Baum-Connes conjecture and K\"{u}nneth formula.

\subsection*{Notation and conventions}

Throughout, if $A$ is a $C^*$-algebra (or more generally, Banach algebra), then $\widetilde{A}$ denotes $A$ itself if $A$ is unital, and the unitization of $A$ if it is not unital.  If $X$ is a subspace of a $C^*$-algebra $A$, then $\widetilde{X}$ is the subspace of $\widetilde{A}$ spanned by $X$ and the unit.  There is an ambiguity here about what happens when $C$ is a $C^*$-subalgebra of $A$, and $C$ has its own unit which is not the unit of $A$: we adopt the convention that in this case, $\widetilde{C}$ means the $C^*$-subalgebra of $A$ generated by $C$ and the unit of $\widetilde{A}$.  This convention will always, and only, apply to $C^*$-subalgebras called $C$, $D$ and $C\cap D$ (plus suspensions and matrix algebras of these), so we hope it causes no confusion.  

We use $1_n$ and $0_n$ to denote the unit and zero element of $M_n(\widetilde{A})$ when it seems helpful to avoid ambiguity, but drop the subscripts whenever things seem more readable without.  We use the usual `top-left corner' identification of $M_n(A)$ with $M_m(A)$ for $n\leq m$, usually without comment.  We also use the usual `block sum' convention that if $a\in M_n(A)$, and $b\in M_m(A)$, then  
$$
a\oplus b:=\begin{pmatrix} a & 0 \\ 0 & b \end{pmatrix} \in M_{n+m}(A).
$$

The symbol $\otimes$ as applied to $C^*$-algebras always denotes the spatial tensor product.  If $X$ is a closed subspace of a $C^*$-algebra $A$, and $B$ is a $C^*$-algebra, then $X\otimes B$ denotes the closure of the algebraic tensor product $X\odot B$ inside $A\otimes B$.  For a $C^*$-algebra $A$, $SA:=C_0(\R)\otimes A$ is its suspension, $S^2A:=S(SA)$ its double suspension, and for a closed subspace $X$ of $A$, $SX:=C_0(\R)\otimes X$.  We always denote the compact operators on $\ell^2(\N)$ by $\mathcal{K}$, so in particular $A\otimes \mathcal{K}$ is the stabilisation of $\mathcal{K}$.

It is typical in $C^*$-algebra $K$-theory to the $K_0$ and $K_1$ groups as generated by certain equivalence classes of projections and unitaries respectively.  However, we will need to work more generally with equivalence classes of idempotents and projections.  This is because one typically has more concrete formulas available in the latter context.  Readers unfamiliar with this approach can find the necessary background in \cite[Chapters II, III and IV]{Blackadar:1998yq}, for example.

We have attempted to keep the paper self-contained and elementary, not assuming much any background beyond basic $C^*$-algebra $K$-theory\footnote{Modulo the comments above about invertibles and idempotents.}.  Although using only elementary language is often desirable in its own right, we must admit that we were also forced into it: indeed, we tried and failed to find `softer', more conceptual, arguments, and would be interested in seeing progress in that direction.

\subsection*{Acknowledgments}

This work was started during a sabbatical visit to the University of M\"{u}nster.  I would like to thank the members of the mathematics department there for their warm hospitality.  

I would like to particularly thank Cl\'{e}ment Dell'Aiera, Dominik Enders, Sabrina Gemsa, Herv\'{e} Oyono-Oyono, Ian Putnam, Aaron Tikuisis, Stuart White, Wilhelm Winter, and Guoliang Yu for numerous enlightening conversations relevant to the topics of this paper.  

The support of the US NSF through grants DMS 1564281 and DMS 1901522 is gratefully acknowledged.

\section{Boundary classes}\label{bound sec}

In this section, we work in the context of general Banach algebras.  This is not needed for our applications, but we hope it clarifies what goes into the results; it also makes no difference to the proofs. 

\begin{definition}\label{iota def}
Let $A$ be a Banach algebra, and let $C$ and $D$ be Banach subalgebras. 
We define maps on $K$-theory by
$$
\iota:K_*(C\cap D)\to K_*(C)\oplus K_*(D),\quad \kappa\mapsto (\kappa,-\kappa).
$$
and 
$$
\sigma:K_*(C)\oplus K_*(D)\to K_*(A), \quad (\kappa,\lambda)\mapsto \kappa+\lambda.
$$
\end{definition}

With notation as above, assume for a moment that $C$ and $D$ are (closed, two-sided) ideals in $A$ such that $A=C+D$.  Then there is a Mayer-Vietoris boundary map $\partial :K_1(A)\to K_0(C\cap D)$ that fits into a long exact sequence
$$
\cdots \stackrel{\iota}{\to} K_1(C)\oplus K_1(D) \stackrel{\sigma}{\to} K_1(A) \stackrel{\partial}{\to} K_0(C\cap D) \stackrel{\iota}{\to} K_0(C)\oplus K_0(D) \stackrel{\sigma}{\to} \cdots.
$$
Our aim in this section is to get analogous results for more general Banach subalgebras $C$ and $D$: for at least some classes $[u]\in K_1(A)$, we want to (non-canonically) construct a `boundary class' $\partial(u)\in K_0(C\cap D)$ that has similar exactness properties with respect to $\iota$ and $\sigma$.  

The next two lemmas concern `almost idempotents'.  We would guess results like these are well-known to experts, but could not find what we needed in the literature.  

\begin{lemma}\label{idem lem}
For any $\epsilon,c>0$ there exists $\delta\in (0,1/16)$ with the following property.  Let $A$ be a Banach algebra and $e\in A$ satisfy $\|e^2-e\|<\delta$ and $\|e\|\leq c$.  Let $\chi$ be the characteristic function of $\{z\in \C\mid \text{Re}(z)>1/2\}$.  Then $\chi(e)$ (defined via the holomorphic functional calculus) is a well-defined idempotent, and satisfies $\|\chi(e)-e\|<\epsilon$.
\end{lemma}

\begin{proof}
First note that if $\delta\in (0,1/16)$ and if $z\in \C$ satisfies $|z^2-z|<\delta$, then $|z||z-1|<\delta$, and so either $|z|<\sqrt{\delta}$, or $|z-1|<\sqrt{\delta}$.  Hence by the polynomial spectral mapping theorem, if $\|e^2-e\|<\delta$, then the spectrum of $e$ is contained in the union of the balls of radius $\sqrt{\delta}$ and centered at $0$ and $1$ respectively.  As $\sqrt{\delta}<1/2$, it follows that $\chi$ is holomorphic on the spectrum of $e$.   Hence $\chi(e)$ makes sense under the assumptions, and is an idempotent by the functional calculus.  

Let now $r=2\sqrt{\delta}<1/2$, and let $\gamma_0$ and $\gamma_1$ be positively oriented circles centered on $0$ and $1$ respectively, and of radius $r$.  Then by the above remarks, if $\|e^2-e\|<\delta$ we have that $\gamma_0\cup \gamma_1$ is a positively oriented contour on which $\chi$ is holomorphic, and that has winding number one around each point of the spectrum of $e$.  Hence by definition of the holomorphic functional calculus 
$$
\chi(e)-e=\frac{1}{2\pi i } \int_{\gamma_0\cup \gamma_1} (\chi(z)-z)(z-e)^{-1}dz.
$$
Estimating the norm of this using that $|\chi(z)-z|=r$ for $z\in \gamma_0\cup \gamma_1$ gives
\begin{equation}\label{chi e minus e}
\|\chi(e)-e\|\leq \frac{1}{2\pi} \int_{\gamma_0\cup\gamma_1} r\|(z-e)^{-1}\||dz|.
\end{equation}

Let us estimate the term $\|(z-e)^{-1}\|$ for $z\in \gamma_0\cup \gamma_1$.  Set $w=1-z$.  Then we have that $w-e$ is also invertible, and 
\begin{align}\label{z minus e}
\|(z-e)^{-1}\| & =\|(w-e)(w-e)^{-1}(z-e)^{-1}\| \nonumber \\ & \leq (c+|w|)\||((z^2-z)-(e^2-e))^{-1}\| \nonumber \\ & \leq (c+2)\||((z^2-z)-(e^2-e))^{-1}\|.
\end{align}
Now, we have that for $z\in \gamma_0\cup \gamma_1$, 
$$
|z^2-z|=|z||z-1|\geq \frac{1}{2}r=\sqrt{\delta}>\delta>\|e^2-e\|.
$$
Hence using the Neumann series inverse formula 
$$
((z^2-z)-(e^2-e))^{-1}=\frac{1}{z^2-z}\Big(1-\frac{e^2-e}{z^2-z}\Big)^{-1}=\frac{1}{z^2-z}\sum_{n=0}^\infty \Big(\frac{e^2-e}{z^2-z}\Big)^n
$$
we get the estimate
$$
\|((z^2-z)-(e^2-e))^{-1}\|\leq \frac{1}{|z^2-z|-\|e^2-e\|}\leq \frac{1}{\frac{1}{2}r-\delta}=\frac{1}{\sqrt{\delta}-\delta}.
$$
Combining this with line \eqref{z minus e}, we see that for $z\in \gamma_0\cup \gamma_1$, 
$$
\|(z-e)^{-1}\|\leq \frac{c+2}{\sqrt{\delta}-\delta}.
$$

To complete the proof, substituting the above estiumate into line \eqref{chi e minus e} gives that 
\begin{align*}
\|\chi(e)-e\| & \leq \frac{1}{2\pi} \int_{\gamma_0\cup\gamma_1} \frac{r(c+2)}{\sqrt{\delta}-\delta}|dz| =\frac{1}{2\pi} \big(\text{Length}(\gamma_0)+\text{Length}(\gamma_1)\big)\frac{r(c+2)}{\sqrt{\delta}-\delta}.
\end{align*}
Substituting in $\text{Length}(\gamma_0)=\text{Length}(\gamma_1)=2\pi r$ and $r=2\sqrt{\delta}$ we get 
$$
\|\chi(e)-e\|\leq \frac{4\sqrt{\delta}(c+2)}{1-\sqrt{\delta}},
$$
which is enough to complete the proof.
\end{proof}

\begin{definition}\label{eps in}
Let $A$ be a Banach algebra, let $X$ be a subset of $A$, let $a\in A$, and let $\epsilon> 0$.  The element $a$ is \emph{$\epsilon$-in} $X$, denoted $a\ine X$, if there exists $x\in X$ with $\|a-x\|\leq  \epsilon$. 
\end{definition}

\begin{lemma}\label{eps in k}
Let $A$ be a Banach algebra and $B$ a Banach subalgebra.  Then for all $c>0$ and all $\epsilon\in (0,\frac{1}{4c+6})$ there exists $\delta>0$ with the following property.
\begin{enumerate}[(i)]
\item Say $n\geq 1$ and say $e\in M_n(A)$ is an idempotent which is $\delta$-in $M_n(B)$ and such that $\|e\|\leq c$.  Then there is an idempotent $f\in M_n(B)$ with $\|e-f\|<\epsilon$.  Moreover, the class $[f]\in K_0(B)$ does not depend on the choice of $\epsilon$, $\delta$, or $f$.
\item Assume moreover that $A$ is unital, and that $B$ contains the unit.  Say $u\in M_n(A)$ is an invertible which is $\delta$-in $M_n(B)$ and such that $\|u^{-1}\|\leq c$.  Then there exists an invertible $v\in M_n(B)$ with $\|u-v\|<\epsilon$, and the class $[v]\in K_1(B)$ does not depend on the choice of $\epsilon$, $\delta$, or $v$. 
\end{enumerate}
\end{lemma}

\begin{proof}
Let $\delta>0$, to be chosen depending on $c$ and $\epsilon$ in a moment, and assume that $e$ is $\delta$-in $M_n(B)$ so there is $b\in M_n(B)$ with $\|b-e\|<\delta$.  Then  
$$
\|b^2-b\|\leq \|e\|\|b-e\|+\|b\|\|b-e\|+\|b-e\|\leq (2c+\delta+1)\delta.
$$
Let $\chi$ be the characteristic function of the half-plane $\{z\in \C\mid \text{Re}(z)>1/2\}$.  Then for suitably small $\delta$ (depending only on $c$ and $\epsilon$), we may apply Lemma \ref{idem lem} to get that $\|b-\chi(b)\|<\epsilon/2$.  Setting $f=\chi(b)$ and assuming also that $\delta<\epsilon/2$ we get that 
$$
\|e-f\|\leq \|e-b\|+\|b-f\|<\epsilon
$$
as desired.  

To see that $[f]\in K_0(B)$ does not depend on the choice of $f$, let $f'\in M_n(B)$ be another idempotent with $\|e-f'\|<\epsilon$.  Then $\|f-f'\|<2\epsilon<1/(2c+3)$.  As $\|f\|\leq c+1$, we see that 
$$
\|f-f'\|<\frac{1}{2c+3}\leq \frac{1}{\|2f-1\|},
$$
whence \cite[Proposition 4.3.2]{Blackadar:1998yq} implies that $f$ and $f'$ are similar, and so in particular define the same K-theory class.

For part (ii), let $\epsilon_0=\frac{1}{4c}$, let $\epsilon\in (0,\epsilon_0]$, and let $\delta=\epsilon$.  Choose any $v\in M_n(B)$ with $\|u-v\|<\delta$.  Then 
$$
\|1-u^{-1}v\|=\|u^{-1}(u-v)\|\leq \|u^{-1}\|\|u-v\|< c\delta=1/4.
$$
Hence $u^{-1}v$ is invertible, and so $v$ is invertible too.  Moreover, estimating the norm of $(u^{-1}v)^{-1}$ using the series expression $(u^{-1}v)^{-1}=\sum_{n=0}^\infty (1-u^{-1}v)^n$ gives that $\|v^{-1}u\|\leq 2$, whence $\|v^{-1}\|=\|v^{-1}uu^{-1}\|\leq 2c$.
On the other hand, if $v'$ also satisfies $\|u-v'\|<\epsilon_0$, then $\|v-v'\|<2\epsilon_0$, and so 
$$
\|1-v^{-1}v'\|\leq \|v^{-1}\|\|v-v'\|<4c\epsilon_0=1.
$$
Hence $v^{-1}v'=e^z$ for some $z\in M_n(B)$ (see for example \cite[II.1.5.3]{Blackadar:2006eq}), and so $\{ve^{tz}\}_{t\in [0,1]}$ is a homotopy between $v$ and $v'$ passing through invertibles in $M_n(B)$, giving that $[v]=[v']$ in $K_1(B)$.
\end{proof}

\begin{definition}\label{eps in k def}
Let $c>0$, let $\epsilon \in(0,\frac{1}{4c+6})$, and let $\delta>0$ be as in Lemma \ref{eps in k}.  Let $A$ be a Banach algebra, and $B$ be a Banach subalgebra of $A$.
\begin{enumerate}
\item Say $e\in M_n(A)$ is an idempotent that is $\delta$-in $M_n(B)$.  Then we write $\{e\}_B\in K_0(B)$ for the class of any idempotent $f\in M_n(B)$ with $\|e-f\|<\epsilon$.  
\item Say $u\in M_n(\widetilde{A})$ is an invertible that is $\delta$-in $M_n(\widetilde{B})$.  Then we write $\{u\}_B\in K_1(B)$ for the class of any invertible $v\in M_n(\widetilde{B})$ with $\|u-v\|<\epsilon$.  
\end{enumerate}
\end{definition}


The next definition is the key technical point that we need to construct our boundary classes.

\begin{definition}\label{lift def}
Let $c>0$, let $\epsilon \in(0,\frac{1}{4c+6})$, and let $\delta>0$ be as in Lemma \ref{eps in k}.   Let $A$ be a Banach algebra, let $C$ and $D$ be Banach subalgebras of $A$, let $u\in M_n(\widetilde{A})$ be an invertible element for some $n$.  An element $v\in M_{2n}(\widetilde{A})$ is a \emph{$(\delta,c,C,D)$-lift} of $u$ if it satisfies the following conditions:
\begin{enumerate}[(i)]
\item \label{lift1} $\|v\|\leq c$ and $\|v^{-1}\|\leq c$;
\item \label{lift2} $v\ind M_{2n}(\widetilde{D})$;
\item \label{lift3} $v\begin{pmatrix} u^{-1} & 0 \\ 0 & u \end{pmatrix}\ind M_{2n}(\widetilde{C})$;
\item \label{lift4} $v\begin{pmatrix} 1 & 0 \\ 0 & 0 \end{pmatrix} v^{-1}\ind M_{2n}(\widetilde{C\cap D})$;
\item \label{lift5} with notation as in Definition \ref{eps in k def}, the $K$-theory class 
$$
\Big\{v\begin{pmatrix} 1 & 0 \\ 0 & 0 \end{pmatrix} v^{-1}\Big\}_{\widetilde{C\cap D}} -\begin{bmatrix} 1 & 0 \\ 0 & 0 \end{bmatrix}\in K_0(\widetilde{C\cap D})
$$ 
is actually in the subgroup $K_0(C\cap D)$.
\end{enumerate}\end{definition}

We may now use such lifts to construct `boundary classes'.

\begin{proposition}\label{bound lem}
Let $c>0$, let $\epsilon \in(0,\frac{1}{4c+6})$.  Then there is $\delta>0$ satisfying the conclusion of Lemma \ref{eps in k},
and with the following properties.   Let $A$ be a Banach algebra, and let $u\in M_n(\widetilde{A})$ be an invertible with $\|u\|\leq c$ and $\|u^{-1}\|\leq c$.  Assume there exist Banach subalgebras $C$ and $D$ of $A$ and a $(\delta,c,C,D)$-lift $v$ of $u$.  Then the $K$-theory class
$$
\partial_v u:=\Big\{v\begin{pmatrix} 1 & 0 \\ 0 & 0 \end{pmatrix} v^{-1}\Big\}_{\widetilde{C\cap D}}-\begin{bmatrix} 1 & 0 \\ 0 & 0 \end{bmatrix}  \in K_0(C\cap D)
$$
has the following properties.
\begin{enumerate}[(i)]
\item If $\iota$ is as in Definition \ref{iota def}, then $\iota(\partial_v u)=0$ in $K_0(C)\oplus K_0(D)$.
\item If $\partial_v u=0$, then there is $l\in \N$ and an invertible $x\ine M_{n+l}(\widetilde{D})$ such that $(u\oplus 1_l)x^{-1}\ine M_{2n}(\widetilde{C})$.  In particular, if $\sigma$ is as in Definition \ref{iota def}, then $\sigma(\{(u\oplus1_l)x^{-1}\}_C,\{x\}_D)=[u]$ in $K_1(A)$.
\end{enumerate}
\end{proposition}
 
\begin{proof}
Let us first consider $\iota(\partial_v u)$.  Note first that as $v$ is $\delta$-in $M_{2n}(\widetilde{D})$, there is $w\in M_{2n}(\widetilde{D})$ such that $\|w-v\|<\delta$.  In particular, $w$ is invertible for $\delta$ suitably small.  It follows by definition of the left hand side that 
$$
\Big\{v\begin{pmatrix} 1 & 0 \\ 0 & 0 \end{pmatrix} v^{-1}\Big\}_{\widetilde{D}}=\Big[w\begin{pmatrix} 1 & 0 \\ 0 & 0 \end{pmatrix} w^{-1}\Big]
$$
in $K_0(\widetilde{D})$ for all suitably small $\delta$.  Hence as elements of $K_0(\widetilde{D})$, 
$$
\Big\{v\begin{pmatrix} 1 & 0 \\ 0 & 0 \end{pmatrix} v^{-1}\Big\}_{\widetilde{D}}-\begin{bmatrix} 1 & 0 \\ 0 & 0 \end{bmatrix}=\Big[w\begin{pmatrix} 1 & 0 \\ 0 & 0 \end{pmatrix} w^{-1}\Big]-\begin{bmatrix} 1 & 0 \\ 0 & 0 \end{bmatrix}.
$$
However, as $w$ is in $M_{2n}(\widetilde{D})$, $\Big[w\begin{pmatrix} 1 & 0 \\ 0 & 0 \end{pmatrix} w^{-1}\Big]=\begin{bmatrix} 1 & 0 \\ 0 & 0 \end{bmatrix}$ in $K_0(\widetilde{D})$, so the above is the zero class in $K_0(\widetilde{D})$, hence also in $K_0(D)$.  

On the other hand, our assumption that $v\begin{pmatrix} u^{-1} & 0 \\ 0 & u \end{pmatrix}$ is $\delta$-in $M_{2n}(\widetilde{C})$ implies similarly that for all $\delta$ suitably small, we have
\begin{align*}
\partial_v u & = \Big\{v\begin{pmatrix} 1 & 0 \\ 0 & 0 \end{pmatrix} v^{-1}\Big\}_{\widetilde{C}}-\begin{bmatrix} 1 & 0 \\ 0 & 0 \end{bmatrix} \\
& =\Big\{\begin{pmatrix} u & 0 \\ 0 & u^{-1} \end{pmatrix}v^{-1}v\begin{pmatrix} 1 & 0 \\ 0 & 0 \end{pmatrix} v^{-1}v\begin{pmatrix} u^{-1} & 0 \\ 0 & u \end{pmatrix}\Big\}_{\widetilde{C}}-\begin{bmatrix} 1 & 0 \\ 0 & 0 \end{bmatrix}, 
\end{align*}
which is zero as a class in $K_0(C)$.  We have shown that the image of $\partial_vu$ in both $K_0(C)$ and $K_0(D)$ is zero, whence $\iota(\partial_v u)=0$ as claimed.

Throughout the rest of the proof, whenever we write `$\delta_n$', it is implicit that this is a positive number, depending only on $c$ and $\delta$, and that tends to zero when $\delta$ tends to zero as long as $c$ stays in a bounded set. 

Now let us assume that $\partial_v u=0$.  This implies that there exists $l\in \N$ and an invertible element $w$ of $M_{2n+l}(\widetilde{C\cap D})$ such that 
$$
\Big\|w\Big(v\begin{pmatrix} 1 & 0 \\ 0 & 0 \end{pmatrix} v^{-1}\oplus 1_l\Big) w^{-1} -  \begin{pmatrix} 1 & 0 \\ 0 & 0 \end{pmatrix} \oplus 1_l\Big\|<\delta_1
$$
for some $\delta_1>0$.  Write $v=\begin{pmatrix} v_{11} & v_{12} \\ v_{21} & v_{22}\end{pmatrix}$, and let 
$$
v_1:=\begin{pmatrix} v_{11} & 0 & v_{12} & 0 \\ 0 & 1_l & 0 & 0 \\ v_{21} & 0 & v_{22} & 0 \\ 0 & 0 & 0 & 1_l\end{pmatrix}\in_\delta M_{n+l+n+l}(\widetilde{D})
$$
(writing the matrix size as $n+l+n+l$ is meant to help understand the size of the various blocks) and if 
$$
w=\begin{pmatrix} w_{11} & w_{12} & w_{13} \\ w_{21} & w_{22} & w_{23} \\ w_{31} & w_{32} & w_{33} \end{pmatrix}\in M_{n+n+l}(\widetilde{C\cap D})
$$ 
let 
$$
w_1:=\begin{pmatrix} w_{11} & 0 & w_{12} & w_{13} \\ 0 & 1_l & 0 & 0 \\ w_{21} &  & w_{22} & w_{23} \\ w_{31} & 0 & w_{32} & w_{33} \end{pmatrix}\in M_{n+l+n+l}(\widetilde{C\cap D}).
$$
Then in $M_{(n+l)+(n+l)}(\widetilde{C})$ we have 
$$
\Big\|w_1v_1\begin{pmatrix} 1 & 0 \\ 0 & 0 \end{pmatrix}v_1^{-1}w_1- \begin{pmatrix} 1 & 0 \\ 0 & 0 \end{pmatrix}\Big\|<\delta_2
$$
for some $\delta_2$.  This implies that for $\delta$ suitably small there exist invertible $x,y\in M_{n+l}(\widetilde{D})$ and $\delta_3$ such that
$$
\Big\|w_1v_1-\begin{pmatrix} x & 0 \\ 0 & y \end{pmatrix}\Big\|<\delta_3.
$$

Now, by assumption 
$$
v\begin{pmatrix} u^{-1} & 0 \\ 0 & u \end{pmatrix}\ind M_{2n}(\widetilde{C}).
$$ 
Write $u_1:=u\oplus 1_l\in M_{n+l}(\widetilde{A})$.  Then 
$$
v_1\begin{pmatrix} u_1^{-1} & 0 \\ 0 & u_1\end{pmatrix}\ind M_{(n+l)+(n+l)}(\widetilde{C}).
$$ 
and thus as $w_1$ is in $M_{2(n+l)}(\widetilde{C})$, we have that 
$$
w_1v_1\begin{pmatrix} u_1^{-1} & 0 \\ 0 & u_1\end{pmatrix}\in_{\delta_4}M_{2(n+l)}(\widetilde{C})
$$
for some $\delta_4$.  Hence in particular, $xu_1^{-1}$ is invertible for $\delta$ suitably small, is $\delta_4$-in $M_{n+l}(\widetilde{C})$, and has norm bounded above by some absolute constant depending only on $c$.  We now have that for $\delta$ suitably small (depending only on $\epsilon$ and $c$), $u_1x^{-1}$ is $\epsilon$-in $M_{n+l}(\widetilde{C})$ and that $x$ is $\epsilon$-in $M_{n+l}(\widetilde{D})$, completing the proof.
\end{proof}

\begin{definition}\label{bound class}
With notation as in Proposition \ref{bound lem}, we call $\partial_v(u)\in K_0(C\cap D)$ the \emph{boundary class} associated to the data $(u,v,C,D)$.
\end{definition}

\section{Approximate ideal structures and the vanishing theorem}\label{decomp bound sec}

Our main goal in this section is to show that approximate ideal structures in Definition \ref{intro decomp} can be used to build lifts as in Definition \ref{lift def}, and thus allow us to build boundary classes.

It would be possible to get analogous results for general Banach algebras, but it would make the statements and proofs more technical.  As our applications are all to the $K$-theory of $C^*$-algebras, at this stage we therefore specialise to that case.

First, it will be convenient to give a technical variation of Definition \ref{intro decomp}.  

\begin{definition}\label{decomp}
Let $A$ be a $C^*$-algebra, let $X\subseteq A$ be a subspace, and let $\delta>0$.  Then a \emph{$\delta$-ideal structure} for $X$ is a triple 
$$
(h,C,D)
$$
consisting of a positive contraction $h$ in the multiplier algebra of $A$, and $C^*$-subalgebras $C$ and $D$ of $A$ such that 
\begin{enumerate}[(i)]
\item \label{decom com prop} $\|[h,x]\|\leq \delta\|x\|$ for all $x\in X$;
\item \label{decom in prop} $hx$ and $(1-h)x$ are $\delta\|x\|$-in $C$ and $D$ respectively for all $x\in X$;
\item \label{decom int in} $h(1-h)x$ and $h^2(1-h)x$ are $\delta\|x\|$-in $C\cap D$ for all $x\in X$.
\end{enumerate}
We say that $A$ has an \emph{approximate ideal structure} over a class $\mathcal{C}$ of pairs of $C^*$-subalgebras if for any $\delta>0$ and finite dimensional subspace $X$ of $A$ there exists a $\delta$-ideal structure $(h,C,D)$ of $X$ with $(C,D)$ in $\mathcal{C}$.
\end{definition}

\begin{remark}\label{silly int rem}
The conditions on multiplying into the intersection in \eqref{decom int in} from Definition \ref{decomp} might look odd for two reasons.  First, they are asymmetric in $h$ and $1-h$: this is a red herring, however, as it would be essentially the same to require that $h(1-h)x$ and $h(1-h)^2x$ are both $\delta\|x\|$-in $C\cap D$.  Second, there are two conditions for $C\cap D$, and only one each for $C$ and $D$.  This seems ultimately attributable to the fact that one needs two polynomials to generate $C_0(0,1)$ as a $C^*$-algebra, but only one each for $C_0(0,1]$ and $C_0[0,1)$. 
\end{remark}

We need to show that admitting an approximate ideal structure bootstraps up to a stronger version of itself (following a suggestion of Aaron Tikuisis and Wilhelm Winter). 

\begin{lemma}\label{tw lem}
Say $A$ is a $C^*$-algebra, $X_0$ is a finite-dimensional subspace of $A$, and $N\geq 2$.  Then there exists a finite-dimensional subspace $X$ of $A$ containing $X_0$, such that for any $\delta>0$ there exists $\delta'>0$ such that if $(h,C,D)$ is a $\delta'$-ideal structure for $X$, then $(h,C,D)$ also satisfies the following properties:
\begin{enumerate}[(i)]
\item \label{decom com prop n} $\|[h,x]\|\leq \delta\|x\|$ for all $x\in X_0$;
\item \label{decom in prop n} for all $n\in \{1,...,N\}$, $h^nx$ (respectively, $h^n(1-h)x$, and $h^n(1-h)x$) is $\delta\|x\|$-in $C$ (respectively $D$, and $C\cap D$) for all $x\in X_0$.
\end{enumerate}
\end{lemma}

\begin{proof}
Take a basis of $X_0$ consisting of contractions, and write each of these as a sum of four positive contractions.  Let $X_1$ be the space of spanned by all these positive contractions, say $\{a_1,...,a_n\}$.  Ley $X$ be spanned by all $m^\text{th}$ roots of all of $a_1,...,a_n$ for $m\in \{1,...,N+1\}$.  Clearly if $\delta'\leq \delta$, then as $Y$ contains $X$, we have the almost commutation property in the statement.

Let us now look at $h^nx$ for $x\in X_0$.  It suffices to look at $h^na$ for some $a\in \{a_1,...,a_n\}$.  Then using the almost commutation property, we have that $h^na$ is close to $(ha^{1/n})^n$, so for $\delta'$ suitably small we get what we want.  Similarly, if $a\in \{a_1,...,a_n\}$, if we write $g=h-1$, then 
$$
h^n(1-h)=(1+g)^n(-g)a=-\sum_{k=0}^n \binom{n}{k} g^{k+1}a,
$$
and again using the almost commutation property, this is close to 
$$
\sum_{k=0}^n \binom{n}{k} (ga^{1/(k+1)})^{k+1},
$$
so we get the right property for $\delta'$ suitably small.  The corresponding property for the intersection is similar, once we realise that for all $n\geq 1$, $h^n(1-h)$ can be written as a polynomial in $h(1-h)$ and $h^2(1-h)$ (proof by induction on $n$, for example): we leave the details of this to the reader.
\end{proof}

The next lemma discusses how approximate ideal structures behave under tensor products.  If $X$ is a subspace of a $C^*$-algebra $A$, recall that we write $X\otimes B$ for the norm closure of the subspace of $A\otimes B$ generated by elementary tensors $x\otimes b$ with $x\in X$ and $b\in B$.

\begin{lemma}\label{exc tens}
Say $A$ is a $C^*$-algebra, and $X$ is a finite-dimensional subspace of $A$.  Then there exists a constant $M_X>0$ depending only on $X$ such that if $(h,C,D)$ is a $\delta$-ideal structure for $X$, and if $B$ is any $C^*$-algebra, then $(h\otimes 1,C\otimes B,D\otimes B)$ is an $M_X\delta$-ideal structure for $X\otimes B$.
\end{lemma}

\begin{proof}
Let $x_1,...,x_n$ be a basis for $X$ consisting of unit vectors, and let $\phi_1,...,\phi_n\in A^*$ be linear functionals dual to this basis, so $\phi_i(x_j)=\delta_{ij}$ (here $\delta_{ij}$ is the Kronecker $\delta$ function).  Let $M=\max_{i=1}^n \|\phi_i\|$.  We claim that $M_X:=nM$ has the property required by the lemma.  Note first that any $a\in X\otimes B$ can be written 
$$
a=\sum_{i=1}^n x_i\otimes b_i
$$
for some unique $b_1,...,b_n\in B$, and that we have for each $i$ 
$$
\|b_i\|=\|(\phi_i\otimes \text{id})(a)\|\leq \|\phi_i\|\|a\|\leq M\|a\|.
$$ 

To see property \eqref{decom com prop}, note that for any $a=\sum_{i=1}^n x_i\otimes b_i\in X\otimes B$ we have
$$
\|[h\otimes 1,a]\|\leq \sum_{i=1}^n \|[h,x_i]\otimes b_i\|\leq \sum_{i=1}^n \delta\|x_i\|\|b_i\|\leq \delta nM\|a\|.
$$
To see properties \eqref{decom in prop} and \eqref{decom int in}, let us look at $ha$ for some $a\in X\otimes B$; the cases of $(1-h)a$, $h(1-h)a$, and $h^2(1-h)a$ are similar. For each each $i\in \{1,...,n\}$ choose $c_i\in C$ with $\|hx_i-c_i\|< \delta$.  Then if $a=\sum_{i=1}^n x_i\otimes b_i \in X\otimes B$ is as above and if $c=\sum_{i=1}^n c_i\otimes b_i\in C\otimes B$ we have 
$$
\|(h\otimes 1)a-c\|\leq \sum_{i=1}^n \|(hx_i-c_i)\otimes b_i\|\leq  \sum_{i=1}^n \delta\|b_i\|\leq \delta nM\|a\|,
$$
which completes the proof.
\end{proof}

\begin{corollary}\label{tp cor}
Say $A$ and $B$ are $C^*$-algebras and $X$ is a finite-dimensional subspace of $A\otimes B$.  Then for any $\delta>0$ there exists a finite-dimensional subspace $Y$ of $A$ and $\delta'>0$ such that if $(h,C,D)$ is a $\delta'$-ideal structure for $Y$, then $(h\otimes 1_B, C\otimes B,D\otimes B)$ is a $\delta$-ideal structure for $X$.
\end{corollary}

\begin{proof}
As the unit sphere of $X$ is compact, there is a finite dimensional subspace $Y$ of $A$ such that for any $x$ in the unit sphere of $X$ there exists $y$ in the unit sphere of $Y\otimes B$ such that $\|y-x\|<\delta/2$.  Let $M_Y$ be as in Lemma \ref{exc tens}, and let $\delta'=\delta/(2M_Y)$.  Lemma \ref{exc tens} implies that if $(h,C,D)$ is a $\delta'$-ideal structure for $Y$ then $(h\otimes 1,C\otimes B,D\otimes B)$ is a $\delta$-ideal structure for $X$.
\end{proof}

For the remainder of this section, we will apply Lemma \ref{exc tens} to tensor products $M_n(A)=A\otimes M_n(\C)$ without further comment.  We will also abuse notation, writing things like `$hu$' for an element $u\in M_n(A)$, when we really mean `$(h\otimes 1_n)u$'.   

The next proposition is the key technical result of this section.  It says that we can use approximate ideal structures to build boundary classes as in Definition \ref{bound class}.  For the statement, recall the notion of a $(\epsilon,c,C,D)$-lift from Definition \ref{lift def} above.

\begin{proposition}\label{v lem}
Let $A$ be a $C^*$-algebra and let $\kappa\in K_1(A)$ be a $K_1$-class.  Then there exist $n$ and an invertible element $u\in M_n(\widetilde{A})$, $c>0$, and a finite-dimensional subspace $X$ of $A$ such that for any $\epsilon>0$ there exists $\delta>0$ such that the following hold.
\begin{enumerate}[(i)]
\item The class $[u]$ equals $\kappa$.
\item If $(h,C,D)$ is a $\delta$-ideal structure of $X$, and if $a=h+(1-h)u$ and $b=h+u^{-1}(1-h)$ then
$$
v:=\begin{pmatrix} 1 & a \\ 0 & 1 \end{pmatrix}\begin{pmatrix} 1 & 0 \\ -b & 1 \end{pmatrix}\begin{pmatrix} 1 & a \\ 0 & 1 \end{pmatrix}\begin{pmatrix} 0 & -1 \\ 1 & 0 \end{pmatrix}
$$ 
is an $(\epsilon,c,C,D)$-lift for $u$.
\end{enumerate}
\end{proposition}

First we have an ancillary lemma.

\begin{lemma}\label{inv cut}
Let $A$ be a $C^*$-algebra and let $u$ be an invertible element of $\widetilde{A}$ such that $u=1+y$ and $u^{-1}=1+z$ with $y,z$ elements of $A$ with norms bounded by some $c>0$.  Let $\delta>0$ and let $h$ be a positive contraction in $M(A)$ such that $\|[h,x]\|\leq \delta\|x\|$ for all $x\in \{y,z\}$.  Define
$$
a:=h+(1-h)u \quad \text{and}\quad b:=h+u^{-1}(1-h).
$$
Then $ba-1$ and $ab-1$ are both within $2(c^2+c)\delta$ of $(y+z)h(1-h)$.
\end{lemma}

\begin{proof}
Using that $y$ and $z$ commute, we have that 
\begin{align*}
[a,b] & =(1-h)yz(1-h)-z(1-h)^2y\\ & =[(1-h),z]y(1-h)+z(1-h)[y,(1-h)] \\ & =[z,h]y(1-h)+z(1-h)[h,y],
\end{align*}
whence $\|[a,b]\|\leq 2c^2\delta$.  Hence it suffices to show that $ab-1$ is within $2c\delta$ of $h(1-h)(y+z)$.  Using that $yz=-y-z$, we see that 
$$
ab-1=(1-h)yh+hz(1-h)
$$
and using that $\|[y,h]\|\leq \delta\|y\|$ and $\|[z,h]\|\leq \delta\|z\|$, we are done.
\end{proof}

\begin{proof}[Proof of Proposition \ref{v lem}] 
Let $u\in M_n(\widetilde{A})$ be any invertible element such that $[u]=\kappa$.  Using that $GL_n(\C)$ is connected, up to a homotopy we may assume that $u$ and $u^{-1}$ are of the form $1+y$ and $1+z$ respectively with $y,z\in M_n(A)$.  Let $X_0$ be the subspace of $A$ spanned by all matrix entries of all monomials of degree between one and three with entries from $\{y,z\}$.  Let $X$ be as in Lemma \ref{tw lem} for this $X_0$ and $N=4$.  Let then $\epsilon>0$ be given, and let $\delta>0$ be fixed, to be determined by the rest of the proof.  Let $(h,C,D)$ be an $\delta$-ideal structure for $X$.

Throughout the proof, anything called `$\delta_n$' is a constant depending on $X$, $\delta$ and $\max\{\|y\|,\|z\|\}$, and with the property that $\delta_n$ tends to zero as $\delta$ tends to zero (assuming the other inputs are held constant).  Note that Lemma \ref{exc tens} implies that there is $\delta_1$ such that $(h,M_n(C),M_n(D))$ is a $\delta_1$-ideal structure of $M_n(A)$ for all $n$.  We check the properties from Definition \ref{lift def}.  Property \eqref{lift1} is clear from the formula for $v$ (which implies a similar formula for $v^{-1}$).

For property \eqref{lift2}, one computes 
\begin{equation}\label{v split}
v=\begin{pmatrix} a(2-ba) & ab-1 \\ 1-ba & b \end{pmatrix}=\begin{pmatrix} a & 0 \\ 0 & b \end{pmatrix} +\begin{pmatrix} a(1-ba) & ab-1 \\ 1-ba & 0 \end{pmatrix}.
\end{equation}
As $a=1+(1-h)y$ and $b=1+z(1-h)$, we have that $a$ and $b$ are both $\delta_2$-in $M_{n}(\widetilde{D})$ for some $\delta_2$.  Hence also $\begin{pmatrix} a & 0 \\ 0 & b \end{pmatrix}$ is $\delta_2$-in $M_{2n}(\widetilde{D})$.  On the other hand, Lemmas \ref{inv cut} and \ref{tw lem} and the choice of $X$ imply that $1-ba$ and $1-ab$ are $\delta_3$-in $M_{2n}(\widetilde{D})$ for some $\delta_3$.  It follows from this and that $a$ is $\delta_2$-in $M_n(\widetilde{D})$ that $\begin{pmatrix} a(1-ba) & ab-1 \\ 1-ba & 0 \end{pmatrix}$ is $\delta_4$-in $M_{2n}(\widetilde{D})$ for some $\delta_4$.  

For part \eqref{lift3}, we compute 
\begin{equation}\label{v u}
v\begin{pmatrix} u^{-1} & 0 \\ 0 & u^{-1} \end{pmatrix} = \begin{pmatrix} au^{-1} & 0 \\ 0 & bu \end{pmatrix}+\begin{pmatrix} a(1-ba)u^{-1} & (ab-1)u \\ (1-ba)u^{-1} & 0 \end{pmatrix}.
\end{equation}
We have that $au^{-1}=1+hz$ and that $\|bu-(1+yh)\|<\delta_5$ for some $\delta_5$.  Hence the first term in line \eqref{v u} is $\delta_6$-in $M_{2n}(\widetilde{C})$ for some $\delta_6$.  For the second term, using Lemma \ref{inv cut} we have that up to some $\delta_7$, $(1-ba)u^{-1}$ and $(ab-1)u$ equal 
$$
(y+z)h(1-h)(1+z)  \quad \text{and} \quad (y+z)h(1-h)(1+y).
$$
On the other hand $\|a(1-ba)u^{-1}-(1+hz)(y+z)h(1-h)\|<\delta_7$ for some $\delta_8$.   The claim follows from all of this and the choice of $X$.

For parts \eqref{lift4} and \eqref{lift5}, note that 
\begin{align*}
v^{-1} & =\begin{pmatrix} 0 & -1 \\ 1 & 0 \end{pmatrix}\begin{pmatrix} 1 & -a \\ 0 & 1 \end{pmatrix}\begin{pmatrix} 1 & 0 \\ b & 1 \end{pmatrix}\begin{pmatrix} 1 & -a \\ 0 & 1 \end{pmatrix} \\ 
& =\begin{pmatrix} b & 1-ba \\ ab-1 & a(2-ba) \end{pmatrix} \\
& =\begin{pmatrix} b & 0 \\ 0 & a \end{pmatrix} +\begin{pmatrix} 0 & 1-ba \\ ab-1 & a(1-ba) \end{pmatrix}.
\end{align*}
Using this and the formula in line \eqref{v split} we have that $v\begin{pmatrix} 1 & 0 \\ 0 & 0 \end{pmatrix}v^{-1}-\begin{pmatrix} 1 & 0 \\ 0 & 0 \end{pmatrix}$ equals 
\begin{equation}\label{v 1 v}
\begin{pmatrix} ab-1 & 0 \\ 0 & 0 \end{pmatrix} +\begin{pmatrix} a(1-ba)b & 0 \\ (1-ba)b & 0 \end{pmatrix}+\begin{pmatrix} 0 & a(1-ba) \\ 0 & 0 \end{pmatrix}+\begin{pmatrix} 0 & a(1-ba)^2 \\ 0 & (1-ba)^2 \end{pmatrix}.
\end{equation}
Now, using Lemma \ref{inv cut} and the fact that $h$ almost commutes with $y$ and $z$, every term appearing is within some $\delta_9$ of something of the form $(1-h)hp_C(h)q_C(y,z)$, where $p_C$ is a polynomial of degree at most $3$ in $h$ (possibly with a constant term), $q_C$ is a noncommutative polynomial of degree at most $3$ with no constant term, and moreover the coefficients in $p_C$ and $q_D$ are universally bounded.  Hence by choice of $X$, all the terms are $\delta_{10}$-in $M_{n}(C\cap D)$, for some $\delta_{10}$.  This completes the proof.
\end{proof}

We are now ready for the proof of Theorem \ref{pre main} from the introduction.

\begin{theorem}
Say that $A$ admits an approximate ideal structure over a set $\mathcal{C}$ such that for all $(C,D)\in \mathcal{C}$, the $C^*$-algebras $C$, $D$, and $C\cap D$ have trivial $K$-theory.  Then $A$ has trivial $K$-theory.  
\end{theorem}

\begin{proof}
It suffices to show that $K_1(A)=K_1(SA)=0$.  For $K_1(A)$, let $\alpha\in K_1(A)$ be an arbitrary class.  Then using Proposition \ref{v lem} we may build a boundary class $\partial_v(u)\in K_0(C\cap D)$.  As $K_0(C\cap D)=0$, this class $\partial_v(u)$ is zero.  Hence by Proposition \ref{bound lem} it is in the image of $\sigma:K_1(C)\oplus K_1(D)\to K_1(A)$.  However, $K_1(C)=K_1(D)=0$ by assumption, so we are done with this case.

The case of $K_1(SA)$ is almost the same.  Indeed, Corollary \ref{tp cor} implies that $SA$ admits an approximate ideal structure over the set $\{(SC,SD)\mid (C,D)\in \mathcal{C}\}$, and we have that $SC$, $SD$, and $SC\cap SD=S(C\cap D)$ all have trivial $K$-theory.
\end{proof}

We remark that Theorem \ref{pre main} can be used to simplify the proof of the main theorem of \cite{Guentner:2014bh}, in particular obviating the need for filtrations and controlled $K$-theory in the proof, and replacing the material of \cite[Section 7]{Guentner:2014bh} entirely.

\section{More on boundary classes}\label{mo bound sec}

In this section we collect together some technical results on boundary classes that are needed for the proof of Theorem \ref{main} on the K\"{u}nneth formula.  We state results for Banach algebras when it makes no difference to the proof, and $C^*$-algebras when the proof is simpler in that case.

 The first result corresponds to exactness at position (I) in line \eqref{pre mv} from the introduction.  For the statement, recall the notion of a $(\delta,c,C,D)$-lift from Definition \ref{lift def}, and the map $\iota:K_0(C\cap D)\to K_0(C)\oplus K_0(D)$ from Definition \ref{iota def}.

\begin{proposition}\label{iota lem}
Let $A$ be a Banach algebra and let $C$ and $D$ be Banach subalgebras of $A$.  Assume that $p,q\in M_n(\widetilde{C\cap D})$ are idempotents such that $[p]-[q]\in K_0(C\cap D)$, and so that $\iota([p]-[q])=0$.  

Then there exist $k\in \N$, an invertible element $u$ of $M_{n+k}(\widetilde{A})$, an invertible element $v$ of $M_{2(n+k)}(\widetilde{A})$, and $c>0$ such that for any $\delta>0$, $v$ and $v^{-1}$ are $(\delta,c,C,D)$-lifts of $u$ and $u^{-1}$ respectively, and such that $\partial_vu=[p]-[q]$ and $\partial_{v^{-1}}(u^{-1})=[q]-[p]$.
\end{proposition}

\begin{proof}
As $\iota([p]-[q])=0$, there exist natural numbers $l\leq k$ and invertible elements $u_C\in M_{n+k}(\widetilde{C})$, $u_D\in M_{n+k}(\widetilde{D})$ such that 
$$
u_C(p\oplus 1_l)u_C^{-1} = q\oplus 1_l = u_D(p\oplus 1_l)u_D^{-1}.
$$
Define
$$
u:=(1-p\oplus 1_l)u_C^{-1}+(p\oplus 1_l)u_D^{-1}\in M_{n+k}(\widetilde{A}).
$$ 
Direct checks that we leave to the reader show that $u$ is invertible with inverse $u^{-1}=u_C(1-p\oplus 1_l)+u_D(p\oplus 1_l)$.  Define now
$$
v:=\begin{pmatrix} (p\oplus 1_l)u_D^{-1} & p\oplus 1_l -1 \\ 1-q\oplus 1_l  & u_D(p\oplus 1_l)\end{pmatrix}\in M_{2(n+k)}(\widetilde{D}).
$$
Note that $v$ is invertible: indeed, direct computations show that
$$
v^{-1}:=\begin{pmatrix} u_D(p\oplus 1_l) & 1-q\oplus 1_l  \\ 1-p_n\oplus 1_l & (p\oplus 1_l)u_D^{-1}\end{pmatrix}.
$$
We also compute that 
\begin{align*}
v\begin{pmatrix} u^{-1} & 0 \\ 0 & u \end{pmatrix} = \begin{pmatrix} p\oplus 1_l & (1-p\oplus 1_l)u_C^{-1} \\  u_C(1-p\oplus 1_l) & q\oplus 1_l \end{pmatrix},
\end{align*}
which is an element of $M_{2(n+k)}(\widetilde{C})$, so at this point we have properties \eqref{lift1}, \eqref{lift2}, and \eqref{lift3} from Definition \ref{lift def}.  

To complete the proof, we compute using the formulas above for $v$ and $v^{-1}$ that 
$$
v\begin{pmatrix} 1 & 0 \\ 0 & 0 \end{pmatrix} v^{-1}=\begin{pmatrix} p\oplus 1_l & 0 \\ 0 & 1-q\oplus 1_l \end{pmatrix},
$$
which is in $M_{2(n+k)}(\widetilde{C\cap D})$.  Moreover, as a class in $K_0(\widetilde{C\cap D})$,
$$
\Big[v\begin{pmatrix} 1 & 0 \\ 0 & 0 \end{pmatrix} v^{-1}\Big]-\begin{bmatrix} 1 & 0 \\ 0 & 0 \end{bmatrix} = [p]-[q],
$$
so in particular this class is in $K_0(C\cap D)$, completing the proof that $v$ satisfies the conditions from Definition \ref{lift def}, and that $\partial_v(u)=[p]-[q]$.

The computations with $v^{-1}$ and $u^{-1}$ replacing $v$ and $u$ are similar: we leave them to the reader.
\end{proof}

The proof of the next lemma consists entirely of direct checks; we leave these to the reader.

\begin{lemma}\label{block lem}
Let $A$ be a Banach algebra, let $c>0$, and let $\epsilon \in(0,\frac{1}{4c+6})$.  Let $\delta>0$ satisfy the conclusion of Proposition \ref{bound lem}.  Assume that for $i\in \{1,...,m\}$, there is an invertible element $u_i\in M_{n_i}(\widetilde{A})$ such that $\|u_i\|\leq c$ and $\|u_i^{-1}\|\leq c$, and let $C$ and $D$ be Banach subalgebras of $A$ such that for each $i$ there is a $(\delta,c,C,D)$-lift $v_i$ of $u_i$.  Let $s\in M_{2(n_1+\cdots +n_m)}$ be the self-inverse permutation matrix defined by the following diagram in the sizes of the matrix blocks
$$
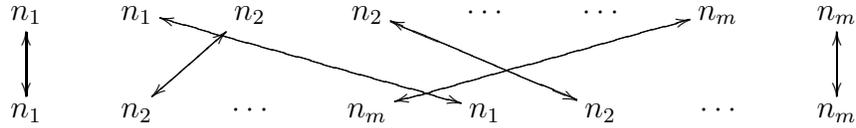
\xymatrix{ n_1 \ar[d] & n_1 \ar[drrr] & n_2 \ar[dl] & n_2 \ar[drr] & \cdots & \cdots  & n_m \ar[dlll] & n_m \ar[d] \\  
n_1 \ar[u] & n_2\ar[ur] & \cdots & n_m \ar[urrr] & n_1\ar[ulll] & n_2 \ar[ull] & \cdots & n_m \ar[u] }
$$
and define 
$$
v_1 \boxplus \cdots \boxplus v_m:=s(v_1\oplus \cdots \oplus v_m)s
$$
Then $v:=v_1 \boxplus \cdots \boxplus v_m$ is a $(\delta,c,C,D)$-lift of $u:=u_1\oplus \cdots \oplus u_m$, and 
$$
\partial_vu=\sum_{i=1}^n \partial_{v_i}(u_i)
$$
in $K_0(C\cap D)$. \qedhere
\end{lemma}

We conclude this section with a technical result on inverses that we will need later.

\begin{lemma}\label{bound lem inv}
Assume that the assumptions of Proposition \ref{v lem} are satisfied.  Then on shrinking $\delta$, we may assume that $v^{-1}$ is also an $(\epsilon,c,C,D)$-lift of $u^{-1}$, and moreover that 
$$
\partial_v (u) = -\partial_{v^{-1}}(u^{-1})
$$ 
as elements of $K_0(C\cap D)$.
\end{lemma}

\begin{proof}
Checking that 
$$
v^{-1}=\begin{pmatrix} 0 & -1 \\ 1 & 0 \end{pmatrix}\begin{pmatrix} 1 & -a \\ 0 & 1 \end{pmatrix}\begin{pmatrix} 1 & 0 \\ b & 1 \end{pmatrix}\begin{pmatrix} 1 & -a \\ 0 & 1 \end{pmatrix}
$$ 
satisfies the properties from Definition \ref{lift def} with respect to $u^{-1}$ is essentially the same as checking the corresponding properties for $v$ and $u$ in the proof of Proposition \ref{v lem}.  We leave the details to the reader.  

It remains to establish the formula $\partial_v (u) = -\partial_{v^{-1}}(u^{-1})$.  For $t\in [0,1]$, define 
$$
v_t:=\begin{pmatrix} 1 & ta \\ 0 & 1 \end{pmatrix}\begin{pmatrix} 1 & 0 \\ -tb & 1 \end{pmatrix}\begin{pmatrix} 1 & ta \\ 0 & 1 \end{pmatrix}\begin{pmatrix} 0 & -1 \\ 1 & 0 \end{pmatrix}.
$$
Analogous computations to those we used to establish to property (iii) in the proof of Proposition \ref{v lem} show that $v_t^{-1}v\begin{pmatrix} 1 & 0 \\ 0 & 0 \end{pmatrix} v^{-1}v_t$ is in $M_{2n}(\widetilde{C\cap D})$ up to an error we can make as small as we like depending on $\delta$ (with $c$ and $X$ fixed), and that the difference 
$$
v_t^{-1}v\begin{pmatrix} 1 & 0 \\ 0 & 0 \end{pmatrix} v^{-1}v_t - v_t^{-1}\begin{pmatrix} 1 & 0 \\ 0 & 0 \end{pmatrix} v_t
$$
is in $M_{2n}(C\cap D)$, again up to an error that we can make as small as we like by making $\delta$ small (and keeping $c$ and $X$ fixed).  Hence for all $t\in [0,1]$ we get that the classes  
$$
\Big\{v_t^{-1}v\begin{pmatrix} 1 & 0 \\ 0 & 0 \end{pmatrix} v^{-1}v_t\Big\}_{\widetilde{C\cap D}} - \Big\{v_t^{-1}\begin{pmatrix} 1 & 0 \\ 0 & 0 \end{pmatrix} v_t\Big\}_{\widetilde{C\cap D}}
$$
of $K_0(C\cap D)$ are well-defined.  They are moreover all the same, as the elements defining them are homotopic.  However, the above equals $\delta_v(u)$ when $t=0$, and equals $-\delta_{v^{-1}}(u^{-1})$ when $t=1$, so we are done.
\end{proof}

\section{Approximate ideal structures and the summation map}\label{sum sec}

In this section, we prove a technical result, based very closely on \cite[Lemma 2.9]{Oyono-Oyono:2016qd}, and corresponding to exactness at position (III) in line \eqref{pre mv} from the introduction.  

The precise statement is a little involved, but roughly it says that given a finite-dimensional subspace $X$ of $A$ there is $\delta>0$ such that if $(h,C,D)$ is a $\delta$-ideal structure for $X$ as in Definition \ref{decomp}, then the maps $\sigma$ and $\iota$ from Definition \ref{iota def} have the following exactness property: if $(\kappa,\lambda)\in K_0(C)\oplus K_0(D)$ is such that $\sigma(\kappa,\lambda)=0$ \emph{and the subspace $X$ contains a `reason' for this element being zero}, then $(\kappa,\lambda)$ is in the image of $\iota$.  

This result is weak: it seems the quantifiers are in the wrong order for it to be useful, meaning that one would like to be able to choose $X$ based on $C$ and $D$, but the statement of the result is the other way around.  Nonetheless, the result is useful, and plays a crucial role in the proof of the injectivity half of theorem \ref{main}.

For the proof of the result, we need a condition that is closely related to the so-called `CIA property' as used in the definition of `nuclear Mayer-Vietoris pairs' in \cite[Definition 4.8]{Oyono-Oyono:2016qd}.  For the statement, let us say that a function $f:(0,\infty)\to (0,\infty)$ is a \emph{decay function} if $f(t)\to 0$ as $t\to 0$.  The following definition is a quantitative variant of Definition \ref{uni ais} from the introduction.

\begin{definition}\label{exc decomp}
Let $(C,D)$ be a pair of $C^*$-subalgebras of a $C^*$-algebra $A$, and let $f$ be a decay function.  Then $(C,D)$ is \emph{$f$-uniform} if for all $C^*$-algebras $B$ and $\delta>0$, if $c\in C\otimes B$ and $d\in D\otimes B$ satisfy $\|c-d\|\leq \delta$, then there exists $x\in (C\cap D)\otimes B$ with $\|x-c\|\leq f(\delta)$ and $\|x-d\|<f(\delta)$.  

Let $A$ be a $C^*$-algebra and $\mathcal{C}$ a set of pairs $(C,D)$ of $C^*$-subalgebras of $A$.  Then $A$ admits a \emph{uniform approximate ideal structure over $\mathcal{C}$} if it admits an approximate ideal structure over $\mathcal{C}$, and if in addition there is a decay function $f$ such that all pairs in $\mathcal{C}$ are $f$-uniform.
\end{definition}

The following example and non-example might help illuminate the definition.  We give some more interesting examples in Appendix \ref{fdc app}.

\begin{example}\label{uniform example}
If $(C,D)$ is a pair of $C^*$-\emph{ideals} in $A$, then $(C,D)$ is $f$-uniform where $f(t)=3t$.  To see this, say that $c\in C\otimes B$ and $d\in D\otimes B$ satisfy $\|c-d\|\leq \delta$.  Let $(h_i)$ be an approximate unit for $C$.  

We claim first that for each $i$, $(h_i\otimes 1_B)d$ is in $(C\cap D)\otimes B$.  Indeed, let $\epsilon>0$, and let $d'$ be an element of the algebraic tensor product $D\odot B$ such that $\|d'-d\|<\epsilon$.  Then $\|(h_i\otimes 1_B)d-(h_i\otimes 1_B)d'\|<\epsilon$, and $(h_i\otimes 1_B)d'\in(C\cap D)\odot B$.  As $\epsilon$ was arbitrary, $(h_i\otimes 1_B)d$ is in $(C\cap D)\otimes B$.  

Choose $i$ large enough so that $\|(h_i\otimes 1_B)c-c\|<\delta$, and set $x=(h_i\otimes 1_B)d$.  Then $x$ is in $(C\cap D)\otimes B$ by the claim, and 
$$
\|x-c\|\leq \|(h_i\otimes 1_B)d-(h_i\otimes 1_B)c\|+\|(h_i\otimes 1_B)c-c\|<2\delta
$$
and 
$$
\|x-d\|\leq \|(h_i\otimes 1_B)d-(h_i\otimes 1_B)c\|+\|c-(h_i\otimes 1_B)c\|+\|c-d\|< 3\delta,
$$
completing the argument that $(C,D)$ is $f$-uniform.
\end{example}

On the other hand, the following non-example shows that $f$-uniformity is quite a strong condition: while it is automatic for ideals by the above, it can fail badly for very simple examples of hereditary subalgebras.

\begin{example}\label{uniform non example}
Let $A=\mathcal{K}$ be the compact operators on $H=\ell^2(\N)$.  Choose projections $p$ and $q$ on $H$ whose ranges have trivial intersection, but such that there are sequences $(x_n)$ and $(y_n)$ of unit vectors in the ranges of $p$ and $q$ respectively with $\|x_n-y_n\|\to 0$ (it is not too difficult to see that such projections exist).  Let $C=p\mathcal{K}p$ and $D=q\mathcal{K}q$, so $C$ and $D$ are hereditary subalgebras of $\mathcal{K}$.  As $\text{range}(p)\cap \text{range}(q)=\{0\}$, we have that $C\cap D=\{0\}$: indeed, any self-adjoint element of $C\cap D$ is a self-adjoint compact operator with all its eigenvectors contained in $\text{range}(p)\cap \text{range}(q)$.  On the other hand, if $p_n$ and $q_n$ are the rank-one projections onto the spans of $x_n$ and $y_n$ respectively, then $p_n\in C$ and $q_n\in D$ for all $n$, and $\|p_n-q_n\|\to 0$.  It follows that the pair $(C,D)$ is not $f$-uniform for any decay function $f$.
\end{example}

The following lemma is immediate from the associativity of the minimal $C^*$-algebra tensor product.

\begin{lemma}\label{tens lem}
Say $A$ and $B$ are $C^*$-algebras and $f$ is a decay function.  If $(C,D)$ is an $f$-uniform pair for $A$, then $(C\otimes B,D\otimes B)$ is an $f$-uniform pair for $A\otimes B$.  \qed
\end{lemma}

We need two preliminary lemmas before we get to the main result.  Recall first that if $u$ is an invertible element of a unital ring, then we have the `Whitehead formula' 
\begin{equation}\label{whitehead0}
\begin{pmatrix} u & 0 \\ 0 & u^{-1}\end{pmatrix}=\begin{pmatrix} 1 & u \\ 0 & 1 \end{pmatrix}\begin{pmatrix} 1 & 0 \\ -u^{-1} & 1 \end{pmatrix} \begin{pmatrix} 1 & u \\ 0 & 1 \end{pmatrix} \begin{pmatrix} 0 & -1 \\ 1 & 0 \end{pmatrix}.
\end{equation}
This implies that invertible elements of the form $\begin{pmatrix} u & 0 \\ 0 & u^{-1}\end{pmatrix}$ are equal to zero in $K$-theory for purely `algebraic' reasons (compare \cite[Lemma 2.5 and Lemma 3.1]{Milnor:1971kl}).  The following lemma can thus be thought of as saying that any invertible element $u$ of a $C^*$-algebra that is zero in $K_1$ for `topological reasons' (i.e.\ is homotopic to the identity) is also zero in $K_1$ for `algebraic' reasons, up to an arbitrarily good approximation\footnote{This cannot be exactly true -- otherwise the algebraic and topological $K_1$ groups of a $C^*$-algebra would always be the same.}.

For the statement of the lemma, recall the notion of being $\delta$-in a subspace of a $C^*$-algebra from Definition \ref{eps in}.

\begin{lemma}\label{hom to alg}
Let $c,\epsilon>0$.  Then there exists $\delta>0$ with the following property.  Let $X$ be a subspace of a $C^*$-algebra $A$ and let $\{u_t\}_{t\in [0,1]}$ be a homotopy of invertibles in $M_n(\widetilde{A})$ such that:
\begin{enumerate}[(i)]
\item $u_1=1_n$;
\item for each $t$, both $u_t$ and $u_t^{-1}$ are $\delta$-in $\{1+x\in M_n(\widetilde{A})\mid  x\in M_n(X)\}$;
\item for each $t$, $\|u_t\|\leq c$ and $\|u_t^{-1}\|\leq c$.
\end{enumerate}

Then there exists $m\in \N$ and invertible elements $a\ind \{1+x\in M_{mn}(\widetilde{A})\mid x\in M_{mn}(X)\}$ and $b\ind \{1+x\in M_{(m+1)n}(\widetilde{A})\mid x\in M_{(m+1)n}(X)\}$ such that $a$, $b$, $a^{-1}$ and $b^{-1}$ all have norm at most $c$, and such that the difference
$$
\begin{pmatrix} u_0 & 0 \\ 0 &  1_{(2m+1)n} \end{pmatrix} - \begin{pmatrix} 1_n & 0 & 0 & 0 \\ 0 & a & 0 & 0 \\ 0 & 0 & a^{-1} & 0 \\ 0 & 0 & 0 & 1_n \end{pmatrix} \begin{pmatrix} b & 0 \\ 0 & b^{-1}\end{pmatrix}
$$
in $M_{2(m+1)n}(\widetilde{A})$ has norm at most $\epsilon$.
\end{lemma}

\begin{proof}
Let $\delta>0$ (to be chosen later), and choose a partition $0=t_0<...<t_m=1$ of the interval $[0,1]$ with the property that for any $i$, $\|u_{t_{i+1}}-u_{t_i}\|<\delta$.  Define 
$$
a:= \begin{pmatrix} u_{t_1}^{-1} & 0 & \hdots & 0 \\ 0 & u_{t_2}^{-1} & \hdots & 0 \\ \vdots & \vdots & \ddots & \vdots \\ 0 & 0 & \hdots & u_{t_m}^{-1} \end{pmatrix}\in M_{mn}(\widetilde{A}).
$$
$$
b:=\begin{pmatrix} u_{t_0} & 0 & \hdots & 0 \\ 0 & u_{t_1}& \hdots & 0 \\ \vdots & \vdots & \ddots & \vdots \\ 0 & 0 & \hdots & u_{t_m} \end{pmatrix}\in M_{(m+1)n}(\widetilde{A}).
$$
Then we have  that 
$$
\begin{pmatrix} u_0 & 0 \\ 0 &  1_{(2m+1)n} \end{pmatrix}  - \begin{pmatrix} 1_n & 0 & 0 & 0 \\ 0 & a & 0 & 0 \\ 0 & 0 & a^{-1} & 0 \\ 0 & 0 & 0 & 1_n \end{pmatrix} \begin{pmatrix} b & 0 \\ 0 & b^{-1}\end{pmatrix} 
$$ 
equals 
$$\begin{pmatrix} 0_{(m+1)n} & 0 & 0 & \cdots & 0  \\ 0 &  1-u_{t_1}u_{t_0}^{-1} & 0 &  \cdots & 0 \\ 0 & 0 & 1-u_{t_2}u_{t_1}^{-1} & \cdots & 0 \\ \vdots & \vdots & \vdots & \ddots & \vdots \\ 0 & 0 & 0 & \cdots & 1-u_{t_m}^{-1}  \end{pmatrix} .
$$
Recalling that $u_{t_m}=1$, the latter element has norm bounded above by 
$$
\max_{i} \|1-u_{t_{i+1}}u_{t_i}^{-1}\|=\max_i \|u_{t_i}-u_{t_{i+1}}\|\|u_{t_i}^{-1}\|<\delta c,
$$
which we can make as small as we like by decreasing the size of $\delta$. 
\end{proof}

The next lemma uses decompositions and the identity in line \eqref{whitehead0} to split up an element of the form $\begin{pmatrix} a & 0 \\ 0 & a ^{-1} \end{pmatrix}$ using approximate ideal structures as in Definition \ref{decomp}.

\begin{lemma}\label{whitehead}
Say $A$ is a $C^*$-algebra and $X$ a finite-dimensional subspace of $A$.    Then there is a finite-dimensional subspace $Y$ of $A$ such that for any $\epsilon>0$ there exists $\delta>0$ so that the following holds.   Assume that $a\in M_n(\widetilde{A})$ is an invertible element such that $a$ and $a^{-1}$ have norm at most $c$, and are $\delta$-in the set $\{1+x\in M_n(\widetilde{A})\mid x\in M_n(X)\}$.   Assume that $(h,C,D)$ is a $\delta$-ideal structure for $Y$.  Then there are homotopies $\{v^C_t\}_{t\in [0,1]}$ and $\{v^D_t\}_{t\in [0,1]}$ of invertible elements such that:
\begin{enumerate}[(i)]
\item for each $t$, $v_t^C\ine \{1+c\mid c\in M_{2n}(C)\}$ and $v_t^D\ine\{1+d\mid d\in M_{2n}(D)\}$; 
\item  $\begin{pmatrix} a & 0 \\ 0 & a^{-1}\end{pmatrix} =v^C_0v^D_0$;
\item $v^C_1=v^D_1=1_{2n}$;
\item for each $t$ the norms of $v^C_t$ and $v^D_t$ are both at most $(3+c)^{5}$.
\end{enumerate}
\end{lemma}

\begin{proof}
Let $Y_0$ be the subspace of $A$ spanned by all monomials of degree between one and four with entries from $X$.  Let $Y$ be as in Lemma \ref{tw lem} for this $Y_0$ and $N=4$.  Let then $\epsilon>0$ be given, and let $\delta>0$ be fixed, to be determined by the rest of the proof.  Let $(h,C,D)$ be a $\delta$-ideal structure for $X$.

Write $a=1+x$ and $a^{-1}=1+y$ with $x,y\ind M_n(X)$.  Consider the product decomposition
\begin{equation}\label{whitehead}
\begin{pmatrix} a & 0 \\ 0 & a^{-1} \end{pmatrix} =\begin{pmatrix} 1 & a \\ 0 & 1 \end{pmatrix}\begin{pmatrix} 1 & 0 \\ -a^{-1} & 1 \end{pmatrix} \begin{pmatrix} 1 & a \\ 0 & 1 \end{pmatrix} \begin{pmatrix} 0 & -1 \\ 1 & 0 \end{pmatrix}.
\end{equation}
Set $x^C:=1+hx$ and $x^D:=(1-h)x$, so that $x^C+x^D=a$.  Similarly, set $y^C:=1+hy$ and $y^D=(1-h)y$, so that $y^C+y^D=a^{-1}$.  For any element $z$ of a $C^*$-algebra, set
$$
X(z):=\begin{pmatrix} 1 & z \\ 0 & 1 \end{pmatrix} \quad \text{and}\quad Y(z):=\begin{pmatrix} 1 & 0 \\ z & 0 \end{pmatrix}.
$$
Then using that $X(z_1+z_2)=X(z_1)X(z_2)$ and similarly for $Y$, the product in line \eqref{whitehead} equals
\begin{align*}
X(x^D)X(x^C)Y(-y^C)Y(-y^D)X(x^C)X(x^D)\begin{pmatrix} 0 & -1 \\ 1 & 0 \end{pmatrix}.
\end{align*}
Rewriting further, this equals the product of 
$$
v^C:=X(x^D)X(x^C)Y(-y^C)X(x^C)\begin{pmatrix} 0 & -1 \\ 1 & 0 \end{pmatrix} X(-x^D),
$$
and 
$$
v^D:=X(x^D)\begin{pmatrix} 0 & 1 \\ -1 & 0 \end{pmatrix} X(-x^C)Y(-y^D)X(x^C)X(x^D)\begin{pmatrix} 0 & -1 \\ 1 & 0 \end{pmatrix}.
$$
We claim this $v^C$ and $v^D$ have the properties required of $v^C_0$ and $v^D_0$ in the statement.  The norm estimates are clear, as is the equation $\begin{pmatrix} a & 0 \\ 0 & a^{-1}\end{pmatrix} =v^C_0v^D_0$.  For the remainder of the proof, any constant called $\delta_n$ depends only on $c$, $X$, and $\delta$, and tends to zero as $\delta$ tends to zero (with the other inputs held constant).

We first claim that $v^C$ is $\epsilon$-in the set $\{1+c\mid c\in M_{2n}(C)\}$ for $\delta$ suitably small.  Using Lemma \ref{exc tens}, $h$ commutes with $x$ and $y$ up to some error $\delta_1$.  Using this, plus the fact that $xy=yx=-y-x$, one computes that
$$
X(x^C)Y(-y^C)X(x^C)\begin{pmatrix} 0 & -1 \\ 1 & 0 \end{pmatrix} 
$$
is within some $\delta_2$ of an element of the form
$$
\begin{pmatrix} 1 & 0 \\ 0 & 1 \end{pmatrix} +\begin{pmatrix} h & 0 \\ 0 & h \end{pmatrix}Z_1,
$$ 
where all entries of $Z_1$ are products of a noncommutative polynomial in $x$ and $y$ of degree at most two and with no constant term, with a polynomial in $h$ of degree at most two.  Hence up to error some $\delta_3$, we have that $v^C$ agrees with 
$$
X(x^D)\Big( \begin{pmatrix} 1 & 0 \\ 0 & 1 \end{pmatrix} +\begin{pmatrix} h & 0 \\ 0 & h \end{pmatrix}Z\Big)X(-x^D),
$$ 
and that up to some $\delta_4$, this is the same as 
$$
\begin{pmatrix} 1 & 0 \\ 0 & 1 \end{pmatrix} +\begin{pmatrix} h & 0 \\ 0 & h \end{pmatrix}Z_2,
$$
where every entry of $Z_2$ is a product of a noncommutative polynomial in $x$ and $y$ of degree at most four and with no constant term, with a polynomial in $h$ of degree at most four.  The claim follows from this, and the choice of $X$.

The computations showing that $v^D$ is $\epsilon$-in the set $\{1+d\mid d\in M_{2n}(D)\}$ for $\delta$ suitably small are similar.  Indeed, we first we note that 
$$
Y(-y_D)=\begin{pmatrix} 1 & 0 \\ 0 & 1 \end{pmatrix} +\begin{pmatrix} 1-h & 0 \\ 0 & 1-h \end{pmatrix}\begin{pmatrix} 0 & 0 \\ -y & 0 \end{pmatrix},
$$
whence $X(-x^C)Y(-y^D)X(x^C)$ is within $\delta_5$ of an element of the form
$$
\begin{pmatrix} 1 & 0 \\ 0 & 1 \end{pmatrix} +\begin{pmatrix} 1-h & 0 \\ 0 & 1-h \end{pmatrix}Z_3
$$
where every entry of $Z_3$ is a product of a noncommutative polynomial in $x$ and $y$ of degree at most two and with no constant term, with a polynomial in $h$ of degree at most two.  Hence $X(-x^C)Y(-y^D)X(x^C)X(x^D)$ is within $\delta_6$ of an element of the form
$$
\begin{pmatrix} 1 & 0 \\ 0 & 1 \end{pmatrix} +\begin{pmatrix} 1-h & 0 \\ 0 & 1-h \end{pmatrix}Z_4,
$$
where every entry of $Z_3$ is a product of a noncommutative polynomial in $x$ and $y$ of degree at most three and with no constant term, with a polynomial in $h$ of degree at most three.  The same is true therefore of 
$$
\begin{pmatrix} 0 & 1 \\ -1 & 0 \end{pmatrix} X(-x^C)Y(-y^D)X(x^C)X(x^D)\begin{pmatrix} 0 & -1 \\ 1 & 0 \end{pmatrix}.
$$
We thus get that $v^D$ is within $\delta_7$ of an element of the form  
$$
\begin{pmatrix} 1 & 0 \\ 0 & 1 \end{pmatrix} +\begin{pmatrix} 1-h & 0 \\ 0 & 1-h \end{pmatrix}Z_5,
$$
where every entry of $Z_5$ is a product of a noncommutative polynomial in $x$ and $y$ of degree at most four and with no constant term, with a polynomial in $h$ of degree at most four.  

To construct homotopies with the required properties, define $x^C_t:=1+(1-t)hx$, $x^D_t:=(1-t)(1-h)x$, $y^C_t:=1+(1-t)hy$, and $y^D_t;=(1-t)(1-h)y$.  Define moreover 
$$
v^C_t:=X(x^D_t)X(x^C_t)Y(-y^C_t)X(x^C_t)\begin{pmatrix} 0 & -1 \\ 1 & 0 \end{pmatrix} X(-x^D_t)
$$
and 
$$
v^D_t:=X(x^D_t)\begin{pmatrix} 0 & 1 \\ -1 & 0 \end{pmatrix} X(-x^C_t)Y(-y^D_t)X(x^C_t)X(x^D_t)\begin{pmatrix} 0 & -1 \\ 1 & 0 \end{pmatrix}.
$$
Using precisely analogous computations to those we have already done, one sees that these elements have the claimed properties: we leave the remaining details to the reader. 
\end{proof}

Here is the key technical result of this section.

\begin{proposition}\label{sigma lem}
Let $A$ be a $C^*$-algebra, let $f$ be a decay function, let $\epsilon>0$, let $c>0$, and let $X$ be a finite-dimensional subspace of $A$.  Then there exists a finite-dimensional subspace $Y$ of $A$ and $\delta>0$ with the following property.

Assume that for some $n\in \N$ there is a homotopy $\{u_t\}_{t\in [0,1]}$ of invertible elements in $M_{n}(\widetilde{A})$ with $u_1=1_n$, and such that each $u_t$ and $u_t^{-1}$ are $\delta$-in the set $\{1+x\in M_n(\widetilde{A})\mid x\in M_n(X)\}$, and have norm at most some $c$.  Then if $(h,C,D)$ is a $\delta$-ideal structure for $Y$ with $(C,D)$ $f$-uniform then the following holds.

Say $l<n$ and $u_C \in M_{n-l}(\widetilde{C})$ and $u_D\in M_{n-l}(\widetilde{D})$ are invertible, such that they and their inverses have norm at most $c$, and such that $\|u_0-u_Cu_D\oplus 1_l\|<\delta$.  Then there exists $k\in \N$ and an invertible element $x\in M_k(\widetilde{C\cap D})$ such that if $[x]\in K_1(C\cap D)$ is the corresponding class, then with notation as in Definition \ref{iota def}, 
$$
\iota[x]=([u_C],[u_D])\in K_1(C)\oplus K_1(D).
$$
\end{proposition}

\begin{proof}
Applying Lemma \ref{hom to alg} to the homotopy $\{u_t\}$ we get $m\in \N$ and invertible elements $a\ind \{1+x\in M_n(\widetilde{A})\mid x\in M_{mn}(X)\}$ an $b\ind \{1+x\in M_{(m+1)n}(\widetilde{A})\mid x\in M_n(X)\}$ such that 
$$
\begin{pmatrix} u_0 & 0 \\ 0 &  1_{(2m+1)n} \end{pmatrix} - \begin{pmatrix} 1_n & 0 & 0 & 0 \\ 0 & a & 0 & 0 \\ 0 & 0 & a^{-1} & 0 \\ 0 & 0 & 0 & 1_n \end{pmatrix} \begin{pmatrix} b & 0 \\ 0 & b^{-1}\end{pmatrix}
$$
has norm at most $\delta$.  Let $Y_a$ and $Y_b$ have the properties in Lemma \ref{whitehead} with respect to $a$ and $b$, and let $Y:=Y_a+Y_b$, a finite dimensional subspace of $A$.  Let then $c$ and $\epsilon$ be given, and let $\delta$ be fixed, to be determined by the rest of the proof.  Let $(h,C,D)$ be a $\delta$-ideal structure of $Y$ with $(C,D)$ $f$-uniform.

As usual, throughout the proof any constant called $\delta_n$ depends on $f$, $c$, $Y$, and $\delta$, and tends to zero as $\delta$ tends to zero.  Applying (a very slight variation of) Lemma \ref{whitehead} to $\begin{pmatrix} a & 0 \\ 0 & a^{-1}\end{pmatrix}$ and $\begin{pmatrix} b & 0 \\ 0 & b^{-1}\end{pmatrix}$, we get elements $v^{C,a}_t$ and $v^{D,a}_t$, and $v^{C,b}_t$ and $v^{D,b}_t$ for $t\in [0,1]$ satisfying the conditions there for some $\delta_1$.  Moreover, if we write $v^{C,a}:=v^{C,a}_1$ and similarly for the other terms, then
\begin{align*}
 \begin{pmatrix} 1_n & 0 & 0 & 0 \\ 0 & a & 0 & 0 \\ 0 & 0 & a^{-1} & 0 \\ 0 & 0 & 0 & 1_n \end{pmatrix} \begin{pmatrix} b & 0 \\ 0 & b^{-1}\end{pmatrix} & =v^{D,a}v^{C,a}v^{C,b}v^{D,b} \\ & =\underbrace{v^{D,a}v^{C,a}v^{C,b} (v^{D,a})^{-1}}_{=:v^C}\underbrace{v^{D,a} v^{D,b}}_{=:v^D}.
\end{align*}
Note that $v^C$ and $v^D$ are $\delta_2$-in $M_{2n}(\widetilde{C})$ and $M_{2n}(\widetilde{D})$ respectively, that they define the trivial class in $K_1(C)$ and $K_1(D)$ respectively, and that they and their inverses have norm at most $(3+c)^{20}$. 

Let $u_C$ and $u_D$ have the properties in the statement.  Replacing $u_C$ and $u_D$ by their block sums with $1_l$, we may (for notational simplicity) assume that $l=0$.   Now, we have that $u_Cu_D$ and $v_Cv_D$ are within some $\delta_3$ of each other.  Hence $1-v_C^{-1}u_C$ and $1-v_Du_D^{-1}$ are within some $\delta_4$ of each other.  Applying our $f$-uniformity assumption, there exists an element $y$ in some matrix algebra over $C\cap D$ that is within some $\delta_5$ of both.   Set $x=1+y$.  Then $x$ is an invertible element of some matrix algebra over $\widetilde{C\cap D}$ (as long as $\delta$ is suitably small) that is close to both $v_C^{-1}u_C$ and to $v_Du_D^{-1}$.  Hence for suitably small $\delta$, we have that as classes in $K_1(C)$
$$
[x]=[v_C^{-1}u_C]=[u_C],
$$
where the second equality follows as $v_C$ represents the trivial class in $K_1(C)$.  Similarly, in $K_1(D)$, 
$$
[x]=[v_Du_D^{-1}]=[u_D^{-1}].
$$
It follows from the last two displayed lines that 
$$
\iota[x]=([u_C],[u_D])
$$
as required. 
\end{proof}

\section{The product map}\label{prod sec}

In this section we recall some facts about the product map
$$
\times:K_*(A)\otimes K_*(B)\to K_*(A\otimes B)
$$
and discuss how it interacts with the boundary classes of Definition \ref{bound class}. 

We first recall concrete formulas for some of the special cases of this product.  See for example \cite[Section 4.7]{Higson:2000bs} for background on this, and \cite[Proposition 4.8.3]{Higson:2000bs} for the particular `$K_1\otimes K_0$' formula that we use.

For each $n$ and $m$, fix an identification $M_n(\C)\otimes M_m(\C)\cong M_{nm}(\C)$ that is compatible with the usual top-left corner inclusions $M_n(\C)\to M_{n+1}(\C)$ as $m$ and $n$ vary.  Use this to identify $M_n(A)\otimes M_m(B)$ with $M_{nm}(A\otimes B)$ for any $C^*$-algebras $A$ and $B$.  Any two such identifications differ by an inner automorphism, so the choice does not matter on the level of $K$-theory.  We will use these identifications without comment from now on.

We recall a basic lemma that is useful for setting up products in the non-unital case: see \cite[Lemma 4.7.2]{Higson:2000bs} for a proof.

\begin{lemma}\label{prod nu}
For a non-unital $C^*$-algebra $A$, let $\epsilon_A:\widetilde{A}\to \C$ denote the canonical quotient map.  For non-unital $C^*$-algebras $A$ and $B$, define $\phi$ to be the $*$-homomorphism 
$$
(\epsilon_A\otimes \text{id}_B)\oplus (\text{id}_A\otimes \epsilon_B):\widetilde{A}\otimes \widetilde{B}\to A\oplus B.
$$
(where we have identified $A\otimes \C$ with $A$ and similarly for $B$ to make sense of this).  Then the map
$$
K_*(A\otimes B)\to K_*(\widetilde{A}\otimes \widetilde{B})
$$
induced by the canonical inclusion $A\otimes B\to \widetilde{A}\otimes \widetilde{B}$ is an isomorphism onto $\text{Kernel}(\phi_*)$. 

Similarly, if $A$ is unital and $B$ is non-unital and $\psi:=\text{id}\otimes \epsilon_B:A\otimes \widetilde{B}\to A$, then the map
$$
K_*(A\otimes B)\to K_*(A\otimes \widetilde{B})
$$
induced by the canonical inclusion $A\otimes B\to A\otimes \widetilde{B}$ is an isomorphism onto $\text{Kernel}(\psi_*)$.  A precisely analogous statement holds if $A$ is non-unital and $B$ is unital.
\qed
\end{lemma}

\begin{definition}\label{prod form}
Let $A$ and $B$ be unital $C^*$-algebras, and let $p\in M_n(A)$ and $q\in M_n(B)$ be idempotents.  Then the \emph{product} of the corresponding $K$-theory classes $[p]\in K_0(A)$ and $[q]\in K_0(B)$ is defined to be
$$
[p]\times [q]:=[p\otimes q]\in K_0(A\otimes B).
$$
Still assuming that $A$ and $B$ are unital, let $u\in M_n(A)$ be invertible and $p\in M_m(B)$ be an idempotent, and define 
$$
u\boxtimes p:=u\otimes p+1\otimes (1-p) \in M_{nm}(A\otimes B).
$$
Note that $u\boxtimes p$ is invertible, with inverse $u^{-1}\boxtimes p$.  The \emph{product} of $[u]\in K_1(A)$ and $[p]\in K_0(B)$ is defined to be 
$$
[u]\times [p]:=[u\boxtimes p]\in K_1(A\otimes B).
$$
One checks that these formulas defined on generators extend to well-defined homomorphisms
$$
\times:K_0(A)\otimes K_0(B) \to K_0(A\otimes B) \quad\text{and}\quad \times:K_1(A)\otimes K_0(B)\to K_1(A\otimes B).
$$

Assume now that $A$ and $B$ are non-unital.  Then one checks that for either $(i,j)=(0,0)$, or $(i,j)=(1,0)$, the canonical composition 
$$
K_i(A)\otimes K_j(B)\to K_i(\widetilde{A})\otimes K_j(\widetilde{B}) \stackrel{\times}{\to} K_{i+j}(\widetilde{A}\otimes \widetilde{B})
$$
takes image in the subgroup $\text{Kernel}(\phi_*)$ of the right hand side, where $\phi$ is as in Lemma \ref{prod nu}.  Using the identification $\text{Kernel}(\phi_*)\cong K_{i+j}(A\otimes B)$ of Lemma \ref{prod nu}, we thus get a general product map
$$
\times : K_i(A)\otimes K_j(B)\to K_{i+j}(A\otimes B)
$$
if $(i,j)\in \{(1,0),(0,0)\}$.  This all works analogously if just one of $A$ or $B$ is non-unital, using the other part of Lemma \ref{prod nu}.  
\end{definition}

For the next definition, for any $C^*$-algebra, let 
$$
\beta^{-1}:K_*(S^2A)\to K_*(A)
$$
be the inverse of the Bott periodicity isomorphism.

\begin{definition}\label{prod part}
Let $A$ and $B$ be $C^*$-algebras.  Define 
$$
K(A)\otimes_1 K(B):=\big( K_1(A)\otimes K_0(B) \big) \oplus \big( K_1(SA)\otimes K_0(SB)\big).
$$
Define a `product' map 
$$
\pi: K(A)\otimes_1 K(B)\to K_1(A\otimes B)
$$
to be the composition 
$$
\begin{array}{lcl}
\big(K_1(A)\otimes K_0(B)\big)\oplus\big( K_1(SA)\otimes K_0(SB)\big) & \stackrel{\times \oplus \times}{\longrightarrow} &  K_1(A\otimes B)\oplus K_1(S^2(A\otimes B)) \\
& \stackrel{\text{id}\oplus \beta^{-1}}{\longrightarrow} & K_1(A\otimes B)\oplus K_1(A\otimes B) \\
& \stackrel{\text{add}}{\longrightarrow} & K_1(A\otimes B) 
\end{array}
$$
We define 
$$
K(A)\otimes_0 K(B):=\big( K_0(A)\otimes K_0(B) \big) \oplus \big( K_0(SA)\otimes K_0(SB)\big)
$$
and 
$$
\pi: K(A)\otimes_0 K(B)\to K_0(A\otimes B)
$$
completely analogously.  
\end{definition}

The product map is natural with respect to suspensions and Bott periodicity.  Hence the map $\pi$ above identifies with the usual product map 
$$
\big(K_1(A)\otimes K_0(B)\big) \oplus \big(K_0(A)\otimes K_1(B)\big) \to K_1(A\otimes B)
$$
under the usual canonical identifications relating suspensions to dimension shifts in $K$-theory, and similarly in the $K_0$ case.

We need a tensor product lemma.  Recall that if $C,D$ are $C^*$-subalgebras of a $C^*$-algebra $A$, and if $B$ is another $C^*$-algebra, then there is a natural inclusion 
$$
(C\cap D)\otimes B\subseteq (C\otimes B)\cap (D\otimes B).
$$
This inclusion need not be an equality above in general: see for example \cite{Kye:1984aa}.  However, $f$-uniform pairs as in Definition \ref{exc decomp} behave well in this setting.

\begin{lemma}\label{int prop}
Let $(C,D)$ be an $f$-uniform pair of $C^*$-algebras of some $C^*$-algebra $A$ for some decay function $f$.  Then the natural inclusion 
$$
(C\cap D)\otimes B\subseteq (C\otimes B)\cap (D\otimes B)
$$
is the identity.
\end{lemma}

\begin{proof}
The assumption of $f$-uniformity directly implies that the image of the inclusion is dense.  The image is a $C^*$-subalgebra, however, so closed.
\end{proof}

HERE

The next lemma is the key technical result of this section.  Morally, it can be thought of as saying that if notation is as in Proposition \ref{bound lem} and if $p$ an idempotent in some matrix algebra over $B$, then the diagram
$$
\xymatrix{ K_1(A)\otimes K_0(B) \ar[r]^-{\partial_v} \ar[d]^-\times & K_0(C\cap D) \ar[d]^-\times \\ K_1(A\otimes B) \ar[r]^-{\partial_{v\boxtimes p}} & K_0((C\cap D)\otimes B) }
$$
makes some sort of sense, and commutes, when one inputs the class $[u]\otimes [p]\in K_1(A)\otimes K_0(B)$.

\begin{lemma}\label{bound com}
Let $A$ be a unital $C^*$-algebra, let $c>0$, and let $\epsilon \in(0,\frac{1}{4c+6})$.  Then there exists $\delta>0$ satisfying the assumptions of Proposition \ref{bound lem}, and with the following additional property.  Assume that $u\in M_n(A)$ is invertible and that $v\in M_{2n}(A)$ is a $(\delta,c,C,D)$-lift for $u$ as in the conclusion of Proposition \ref{bound lem}.  Let $B$ be a $C^*$-algebra, and let $p\in M_m(B)$ be an idempotent with $\|p\|\leq c$. 

Then (with notation as in Definition \ref{prod form}) $v\boxtimes p$ is a $(\epsilon,c,C,D)$-lift for $u\boxtimes p$, and we have
$$
\partial_v(u)\times [p]=\partial_{v\boxtimes p}(u\boxtimes p)
$$
as classes in $K_0((C\cap D)\otimes B)$.
\end{lemma}

\begin{proof}
We leave it to the reader to check that $v\boxtimes p$ is a $(\epsilon,c,C,D)$-lift of $u\boxtimes p$ for suitably small $\delta>0$ (depending only on $\epsilon$ and $c$).  Computing, we see that 
\begin{align*}
&\partial_{v\boxtimes p} (u\boxtimes p) \\ & =\Big\{\big(v\otimes p+1\otimes (1-p)\big)\begin{pmatrix} 1 & 0 \\ 0 & 0 \end{pmatrix} \big(v^{-1}\otimes p +1\otimes (1-p)\big)\Big\}_{(\widetilde{C\cap D})\otimes B}-\begin{bmatrix} 1 & 0 \\ 0 & 0 \end{bmatrix}  \\
& = \Big\{v\begin{pmatrix} 1 & 0 \\ 0 & 0 \end{pmatrix} v^{-1}\otimes p +\begin{pmatrix} 1 & 0 \\ 0 & 0 \end{pmatrix} \otimes (1-p)\Big\}_{(\widetilde{C\cap D})\otimes B}-\begin{bmatrix} 1 & 0 \\ 0 & 0 \end{bmatrix}.
\end{align*}
Using that the two terms inside the curved brackets are orthogonal, we have 
\begin{align*}
\Big\{v\begin{pmatrix} 1 & 0 \\ 0 & 0 \end{pmatrix} v^{-1}\otimes p &  +\begin{pmatrix} 1 & 0 \\ 0 & 0 \end{pmatrix} \otimes (1-p)\Big\}_{(\widetilde{C\cap D})\otimes B} \\ & =\Big\{v\begin{pmatrix} 1 & 0 \\ 0 & 0 \end{pmatrix} v^{-1}\otimes p \Big\}_{(\widetilde{C\cap D})\otimes B}+\Big[\begin{pmatrix} 1 & 0 \\ 0 & 0 \end{pmatrix} \otimes (1-p)\Big].
\end{align*}
As 
$$
\Big[\begin{pmatrix} 1 & 0 \\ 0 & 0 \end{pmatrix} \otimes (1-p)\Big]-\begin{bmatrix} 1 & 0 \\ 0 & 0 \end{bmatrix} = - \Big[\begin{pmatrix} 1 & 0 \\ 0 & 0 \end{pmatrix} \otimes p\Big],
$$
we get that 
\begin{align*}
\partial_{v\boxtimes p} (u\boxtimes p)  &=\Big\{v\begin{pmatrix} 1 & 0 \\ 0 & 0 \end{pmatrix} v^{-1}\otimes p \Big\}_{(\widetilde{C\cap D})\otimes B} -\Big[\begin{pmatrix} 1 & 0 \\ 0 & 0 \end{pmatrix} \otimes p\Big] \\
& = \Bigg(\Big\{v\begin{pmatrix} 1 & 0 \\ 0 & 0 \end{pmatrix} v^{-1} \Big\}_{\widetilde{C\cap D}}-  \begin{bmatrix} 1 & 0 \\ 0 & 0 \end{bmatrix} \Bigg)\times [p],
\end{align*}
which is exactly $\partial_v(u)\times [p]$ as claimed.
\end{proof}

We also need compatibility results for the maps $\iota$ and $\sigma$ of Definition \ref{iota def} and the maps $\pi$ of Definition \ref{prod part}.  These are recorded by the following lemma.

\begin{lemma}\label{iota com}
Let $C$ and $D$ be an excisive pair of $C^*$-subalgebras of a $C^*$-algebra $A$, and let $B$ be a $C^*$-algebra.  Then for $i\in \{0,1\}$, the diagrams
$$
\xymatrix{ K(C\cap D)\otimes_i K(B) \ar[r]^-{\iota\otimes \text{id}} \ar[d]^-\pi & K(C)\otimes_i K(B)\oplus K(D) \otimes_i K(B) \ar[d]^-{\pi} \\ 
K_i((C\cap D)\otimes B) \ar[r]^-\iota & K_1(C\otimes B)\oplus K_1(D\otimes B) }
$$
and 
$$
\xymatrix{ K(C)\otimes_i K(B)\oplus K(D) \otimes_i K(B) \ar[r]^-{\sigma\otimes \text{id}} \ar[d]^-{\pi} & K(A)\otimes_i K(B) \ar[d]^-{\pi} \\ 
K_i(C\otimes B)\oplus K_i(D\otimes B) \ar[r]^-\sigma & K_i(A\otimes B) }
$$
commute (where we have the canonical identification of Lemma \ref{int prop} amongst others to make sense of this).
\end{lemma}

\begin{proof}
This follows directly from naturality of the product maps and Bott maps in $K$-theory.
\end{proof}

\section{The inverse Bott map}\label{bott sec}

For a $C^*$-algebra $A$, let 
$$
\bt:K_*(S^2A)\to K_*(A)
$$
be the inverse Bott isomorphism.  It will be convenient to have a model for $\bt$ based on an asymptotic family.  In this section, we recall some facts about asymptotic families and their action on $K$-theory (in the `naive', rather than $E$-theoretic, picture).  We then discuss how the inverse Bott map can be represented by an asymptotic family with good properties.

Recall (see for example \cite[Definition 1.3]{Guentner:2000fj}) that an \emph{asymptotic family} between $C^*$-algebras $A$ and $B$ is a collection of maps $\{\alpha_t:A\to B\}_{t\in [1,\infty)}$ such that:
\begin{enumerate}[(i)]
\item for each $a\in A$, the map $t\mapsto \alpha_t(a)$ is continuous and bounded;
\item for all $a_1,a_2\in A$ and $z_1,z_2\in \C$, the quantities 
$$
\alpha_t(a_1a_2)-\alpha_t(a_1)\alpha_t(a_2),\quad \alpha_t(a_1^*)-\alpha_t(a_1)^*
$$
and 
$$
\alpha_t(z_1a_1+z_2a_2)-z_1\alpha_t(a_1)-z_2\alpha_t(a_2)
$$
all tend to zero as $t$ tends to infinity. 
\end{enumerate}

An asymptotic family $\{\alpha_t:A\to B\}_{t\in [1,\infty)}$ canonically defines a map $\alpha_*:K_*(A)\to K_*(B)$.  One way to define $\alpha_*$ uses the composition product in $E$-theory and the identification of $E_*(\C,A)$ with $K_*(A)$.  However, there is also a more naive and direct way.  This is certainly very well-known, but we are not sure exactly where to point in the literature for a description, so we describe it here for the reader's convenience.

Assume for simplicity that $A$ and $B$ are not unital (this is the only case we will need), and that $\{\alpha_t:A\to B\}$ is an asymptotic family.  We extend $\{\alpha_t\}$ to unitisations and matrix algebras just as we would for a $*$-homomorphism.  Note that as $A$ and $B$ are not unital, the extended asymptotic morphism on unitisations takes units to units.

If $e\in M_n(\widetilde{A})$ is an idempotent, then $\|\alpha_t(e)^2-\alpha_t(e)\|\to 0$ as $t\to\infty$.  Hence if $\chi$ is the characterisitic function of the half-plane $\{z\in \C\mid \text{Re}(z)>1/2\}$ then $\chi(\alpha_t(e))$ (defined using the holomorphic functional calculus) is a well-defined idempotent in $M_n(B)$ for all $t$ suitably large.  If $[e]-[f]$ is a formal difference of idempotents in $M_n(\widetilde{A})$ defining a class in $K_0(A)$, then one sees that for all $t$ suitably large the formal difference 
$$
[\chi(\alpha_t(e))]-[\chi(\alpha_t(f))]\in K_0(\widetilde{B})
$$
is in the kernel of the natural map $K_0(\widetilde{B})\to K_0(\C)$ induced by the canonical quotient $\widetilde{B}\to \C$.  We define $\alpha_*([e]-[f]):=[\chi(\alpha_t(e))]-[\chi(\alpha_t(f))]$ for any suitably large $t$. The choice of $t$ does not matter, as for any $t'\geq t$, the path $\{\chi(\alpha_s(e))\}_{s\in [t,t']}$ is a homotopy of idempotents, and similarly for $f$.

Similarly (and more straightforwardly), if $u\in M_n( \widetilde{A})$ is invertible, then as the extension of $\alpha_t$ to unitisations is unital, for all suitably large $t$, $\alpha_t(u)\in M_n(\widetilde{B})$ is invertible, and we get a well-defined class $\alpha_*[u]:=[\alpha_t(u)]$ for any suitably large $t$.  In this way, we get a well-defined homomorphism 
$$
\alpha_*:K_*(A)\to K_*(B).
$$

We also need to discuss the tensor product of an asymptotic family and a $*$-homomorphism.  First, we describe how an asymptotic family is essentially the same thing as a $*$-homomorphism $A\to C_b([1,\infty),B) / C_0([1,\infty),B)$.  More precisely, given an asymptotic family $\{\alpha_t:A\to B\}$, we can define
$$
\alpha:A\to \frac{C_b([1,\infty),B)}{C_0([1,\infty),B)}, \quad a\mapsto [t\mapsto \alpha_t(a)].
$$
Conversely, the Bartle-Graves selection theorem implies the existence of a continuous section $s:C_b([1,\infty),B) / C_0([1,\infty),B)\to C_b([1,\infty),B)$.  Then given a homomorphism $\alpha:A\to C_b([1,\infty),B) / C_0([1,\infty),B)$ we can define an asymptotic family $\{\alpha_t:A\to B\}$ by the formula $\alpha_t(a):=s(\alpha(a))(t)$.  If $s$ and $s'$ are two different choices of section and $\{\alpha_t\}$ and $\{\alpha_t'\}$ the corresponding asymptotic families, then $\alpha_t(a)-\alpha_t'(a)\to 0$ as $t\to\infty$ (compare for example \cite[pages 4-5]{Guentner:2000fj}).  In  particular, this implies that the induced maps $\alpha_*$ and $\alpha'_*$ on $K$-theory on the same.  

We may use this correspondence to define the tensor product of an asymptotic family and a $*$-homomorphism.  Say $\{\alpha_t:A\to B\}$ is an asymptotic family, and $\phi:C\to D$ a $*$-homomorphism with $D$ nuclear.  As in \cite[Proposition 4.3]{Guentner:2000fj}, we get a natural $*$-homomorphism
$$
\frac{C_b([1,\infty),B)}{C_0([1,\infty),B)}\otimes D\to \frac{C_b([1,\infty),B\otimes D)}{C_0([1,\infty),B\otimes D)},
$$
where we have used nuclearity of $D$ to see that the spatial tensor product $\cdot \otimes D$ agrees with the maximal tensor product $\cdot \otimes_{\max}D$.  Hence we get a $*$-homomorphism 
\begin{equation}\label{am tp}
A\otimes C\stackrel{\alpha\otimes \phi}{\longrightarrow} \frac{C_b([1,\infty),B)}{C_0([1,\infty),B)}\otimes D\to \frac{C_b([1,\infty),B\otimes D)}{C_0([1,\infty),B\otimes D)}.
\end{equation}

\begin{definition}\label{as fam tp}
We let $\{\alpha_t\otimes \phi:A\otimes C\to B\otimes D\}$ be any choice of asymptotic family corresponding to the $*$-homomorphism in line \eqref{am tp}.
\end{definition}

`The' asymptotic family $\{\alpha_t\otimes \phi\}$ is unfortunately not canonically determined by $\{\alpha_t\}$ and $\phi$.  Nonetheless, any such choice will satisfy 
$$
(\alpha_t\otimes \phi)(a\otimes c)-\alpha_t(c)\otimes \phi(c)\to 0 \quad \text{as} \quad t\to\infty
$$ 
on elementary tensors, and any two such choices will induce the same map $K_*(A\otimes C)\to K_*(B\otimes D)$.

The following lemma is the main technical result of this section.  It says that asymptotic families are compatible with boundary classes as in Definition \ref{bound class}.  For the statement, recall the definition of a $(\delta,c,C,D)$-lift from Definition \ref{lift def}.

\begin{lemma}\label{bott bound}
Let $c,\epsilon>0$.  Then there is $\delta>0$ with the following property.  

Let $\{\alpha_t:A\to B\}$ be an asymptotic family between non-unital $C^*$-algebras, and let $(C_A,D_A)$ be a pair of $C^*$-subalgebras of $A$ and $(C_B,D_B)$ a pair of $C^*$-subalgebras of $B$ such that for all $c\in C_A$ and $d\in D_A$,
$$
d(\alpha_t(c), C_B) \quad \text{and} \quad d(\alpha_t(d),D_B) 
$$
tend to zero as $t$ tends to infinity.  Assume that $u\in M_{2n}(\widetilde{A})$ is an invertible element with $\|u\|\leq c$ and $\|u^{-1}\|\leq c$, and let $v$ be a $(\delta/2,c/2,C_A,D_A)$-lift of $u$.  Then for all suitably large $t$, $\alpha_t(v)\in M_{2n}(\widetilde{B})$ is a $(\delta,c,C_B,D_B)$-lift of $\alpha_t(u)$, and moreover 
$$
\partial_{\alpha_t(v)}(\alpha_t(u))=\alpha_*(\partial_v(u))
$$
in $K_0(C_B\cap D_B)$ for all suitably large $t$.
\end{lemma}

\begin{proof}
We use the same notation $\{\alpha_t\}$ for the canonical extensions to matrix algebras and unitisations.  Note first that as the extension of $\{\alpha_t\}$ to unitisations is unital, and as $\alpha_t$ is asymptotically multiplicaitve, $\alpha_t(u)$ and $\alpha_t(v)$ are invertible for all suitably large $t$.

We first claim that asymptotic families are `asymptotically contractive' in the following sense: for any $a\in A$ and any $\epsilon>0$ we have $\|\alpha_t(a)\|< \|a\|+\epsilon$ for all suitably large $t$.  Indeed, let
$$
\alpha:A\to \frac{C_b([1,\infty),B)}{C_0([1,\infty),B)}, \quad a\mapsto [t\mapsto \alpha_t(a)]
$$
be the corresponding $*$-homomorphism.  As $\alpha$ is a $*$-homomorphism, it is contractive.  Hence by definition of the quotient norm, for any $\epsilon>0$ there is $b\in C_0([1,\infty),B)$ such that 
$$
\sup_{t\in [1,\infty)} \|\alpha_t(a)-b(t)\|<\|\alpha(a)\|+\epsilon\leq \|a\|+\epsilon.
$$
As $\|b(t)\|\to 0$ as $t\to\infty$, the claim follows.

Now, from the claim and the fact that for all $d\in D_A$, $d(\alpha_t(d),D_B)$ tends to zero as $t$ tends to infinity, that we have that $\alpha_t(v)$ is $\delta$-in $M_{2n}(\widetilde{D_B})$ for all suitably large $t$. Similarly, and using also the asymptotic multiplicativity and unitality of $\{\alpha_t\}$, we get that 
$$
\alpha_t(v)\begin{pmatrix} \alpha_t(u)^{-1} & 0 \\ 0 & \alpha_t(u) \end{pmatrix} \ind M_{2n}(\widetilde{C_B})
$$
for all suitably large $t$.  The remaining conditions from Definition \ref{lift def} follow similarly.

To see that $\partial_{\alpha_t(v)}(\alpha_t(u))=\alpha_*(\partial_v(u))$ for $t$ large enough, note that for suitably large $t$, the former is represented by 
\begin{equation}\label{bott bound 1}
\Big\{\alpha_t(v)\begin{pmatrix} 1 & 0 \\ 0 & 0 \end{pmatrix} \alpha_t(v)^{-1} \Big\}_{\widetilde{C_B\cap D_B}}-\begin{bmatrix} 1 & 0 \\ 0 & 0 \end{bmatrix}.
\end{equation}
For the latter, one starts by choosing an idempotent $f\in M_{2n}(\widetilde{C_A\cap D_A})$ suitably close to $v\begin{pmatrix} 1 & 0 \\ 0 & 0 \end{pmatrix} v^{-1}$ as in Lemma \ref{eps in k} so that 
$$
\Big\{v\begin{pmatrix} 1 & 0 \\ 0 & 0 \end{pmatrix} v^{-1} \Big\}_{\widetilde{C_A\cap D_A}}=[f]
$$
in $K_0(\widetilde{C_A\cap D_A})$.  Then $\alpha_*(\partial_v(u))$ is represented by 
\begin{equation}\label{bott bound 2}
\chi(\alpha_t(f)) - \begin{bmatrix} 1 & 0 \\ 0 & 0 \end{bmatrix}
\end{equation}
for $t$ suitably large, where $\chi$ is as usual the characteristic function of $\{z\in \C\mid \text{Re}(z)>1/2\}$.  Now, as $\|\alpha_t(f)\|$ is uniformly bounded in $t$ and as $\|\alpha_t(f)^2-\alpha_t(f)\|\to 0$, we may apply Lemma \ref{idem lem} to conclude that $\|\alpha_t(f)-\chi(\alpha_t(f))\|\to 0$.  On the other hand, by making $\delta$ suitably small and $t$ large, and using the `asymptotic contractiveness' claim at the start of the proof, we can make $\alpha_t(f)$ as close as we like to
$$
\alpha_t(v)\begin{pmatrix} 1 & 0 \\ 0 & 0 \end{pmatrix} \alpha_t(v)^{-1}.
$$
Comparing lines \eqref{bott bound 1} and \eqref{bott bound 2}, the proof is complete.
\end{proof}

We need the fact that Bott periodicity is induced by an appropriate asymptotic morphism.  The following lemma is  well-known.

\begin{lemma}\label{bott lem}
For any $C^*$-algebra $A$ there is an associated asymptotic family 
$$
\alpha_t:S^2A\leadsto A\otimes \mathcal{K}
$$
with the following properties:
\begin{enumerate}[(i)]
\item the map $\alpha_*$ induced on $K$-theory by $\{\alpha_t\}$ is the inverse Bott map $\bt$;
\item if $B$ is a $C^*$-subalgebra of $A$ and $\{\alpha_t^A\}$ and $\{\alpha_t^B\}$ are the asymptotic families associated to $A$ and $B$ respectively, then for all $b\in S^2B$, $\alpha^A_t(b)-\alpha_t^B(b)\to 0$ as $t\to\infty$; 
\item \label{bl fd cont} for any finite-dimensional subspace $X$ of $A$ and any element of $S^2X$, 
$$
\sup\{d(\alpha_t(x),X\otimes \mathcal{K})\mid x\in S^2X,\|x\|\leq 1\}
$$
tends to zero as $t$ tends to infinity;
\item if we fix an inductive limit description $\mathcal{K}=\overline{\bigcup_{n=1}^\infty M_n(\C)}$, then for all $t$ and all $a\in S^2A$, $\alpha_t(a)$ has image in the $*$-subalgebra $\bigcup_{n=1}^\infty M_n(A)$ of $A\otimes \K$.
\end{enumerate} 
\end{lemma}

\begin{proof}
There are several different ways to do this.  We sketch one from \cite{Elliott:1993aa} based on the representation theory of the Heisenberg group.  As in \cite[Section 4]{Elliott:1993aa}, one may canonically construct a continuous field of $C^*$-algebras over $[0,1]$ with the fibre at $0$ equal to $S^2\C$, and all other fibres equal to $\mathcal{K}$.  As explained in \cite[Appendix 2.B]{Connes:1994zh} or \cite[pages 101-2]{Connes:1990if}, such a deformation (non-canonically) gives rise to an asymptotic family $\{\alpha_t:S^2\C\to A\otimes \mathcal{K}\}$, and this family induces the map on $K$-theory described in general in \cite[Section 3]{Elliott:1993aa}, and which is shown in \cite[Theorem 4.5]{Elliott:1993aa} to be the inverse of the Bott periodicity isomorphism.

This gives us our asymptotic family $\{\alpha_t\}$ for the case $A=\C$.  In the general case, we may take $\{\alpha^A_t\}$ to be a choice of asymptotic family $\{\alpha_t\otimes \text{id}_A:S^2\C\otimes A\to \mathcal{K}\otimes A\}$ as in Definition \ref{as fam tp}.

Note that the construction of $\{\alpha_t^A\}$ is not canonical at two places: going from a deformation to an asymptotic family, and taking the tensor product.  However, any two asymptotic families $\{\alpha_t\}$, $\{\alpha_t'\}$ constructed from different choices will satisfy $\alpha_t(a)-\alpha_t'(a)\to 0$ as $t\to\infty$ for all $a\in S^2A$.  It follows that the asymptotic families so constructed satisfy (i), (ii), and (iii).  

To make it also satisfy (iv), let $\{k_t\}_{t\in [1,\infty)}$ be a continuous family of positive contractions in $\bigcup M_n(A)\subseteq \mathcal{K}$ such that for all $k\in \mathcal{K}$, $k_tkk_t-k\to 0$ as $\to\infty$.  For each $a\in A$, choose a homeomorphism $s_a:[1,\infty)\to [1,\infty)$ such that 
$$
\alpha_t^A(a)-(1\otimes k_{s_a(t)})\alpha_t(a)(1\otimes k_{s_a(t)})\to 0
$$
as $t\to\infty$.  Replacing $\alpha_t$ with the map 
$$
a\mapsto (1\otimes k_{s_a(t)})\alpha_t^A(a)(1\otimes k_{s_a(t)}),
$$
we get the result.
\end{proof}

\section{Surjectivity of the product map}\label{surj sec}

In this section, we prove the surjectivity half of Theorem \ref{main}.

\begin{theorem}\label{kun surj}
Let $A$ be a $C^*$-algebra, and say $A$ admits a uniform ideal structure over a class $\mathcal{C}$ such that for each $(C,D)\in \mathcal{C}$, $C$, $D$, and $C\cap D$ satisfy the K\"{u}nneth formula.  Then for any $C^*$-algebra $B$ with free abelian $K$-theory, the product map 
$$
\times:K_*(A)\otimes K_*(B)\to K_*(A\otimes B)
$$
is surjective.
\end{theorem}

\begin{proof}
It suffices to show that the product maps
$$
\pi:K(A)\otimes_0 K(B) \to K_0(A\otimes B) \quad \text{and}\quad \pi:K(A)\otimes_1 K(B) \to K_1(A\otimes B) 
$$
of Definition \ref{prod part} are surjective for any $B$ with $K_*(B)$ free.  Replacing $B$ with its suspension, it moreover suffices to show that the second of the maps above is surjective.  Let then $\kappa$ be an arbitrary class in $K_1(A\otimes B)$.

Let $X\subseteq A\otimes B$ and $u\in M_n(\widetilde{A})$ be as in Proposition \ref{v lem} for this $\kappa$.  Using Corollary \ref{tp cor} and Lemma \ref{tens lem}, for any $\delta>0$ there is an $f$-uniform $\delta$-ideal structure of the form $(h\otimes 1,C\otimes B,D\otimes B)$ for $X$.  Fix such an ideal structure for a very small $\delta>0$ (how small will be determined by the rest of the proof).

Using Proposition \ref{v lem} we may build an element $v\in_{\delta_1} M_{2n}(\widetilde{A\otimes B})$ with the properties stated there, for some constant $\delta_1$ that tends to zero as $\delta$ tends to zero.  We may use $v$ to construct an element $\partial_vu\in K_1((C\cap D)\otimes B)$ as in Proposition \ref{bound lem} (here we use the identification $(C\cap D)\otimes B=C\otimes B\cap D\otimes B$ of Lemma \ref{int prop}), and have that if 
$$
\iota:K_0((C\cap D)\otimes B)  \to K_0(C\otimes B)\oplus K_0(D\otimes B) 
$$ 
is the map from Definition \ref{iota def}, then $\iota(\partial_vu)=0$.  

Using that the product map $\pi$ for $C\cap D$ is surjective, we may lift $\partial_vu$ to an element $\lambda$ of $K(C\cap D)\otimes_1 K(B)$.  With notation as in Definition \ref{prod part}, Lemma \ref{iota com} gives that the diagram
$$
\xymatrix{ K(C\cap D) \otimes_0 K(B) \ar[r]^-{\iota\otimes \text{id}} \ar[d]^-\pi &  K(C)\otimes_0 K(B) \oplus K(D)\otimes_0 K(B) \ar[d]^-\pi \\ K_0((C\cap D)\otimes B)  \ar[r]^-\iota & K_0(C\otimes B)\oplus K_0(D\otimes B) }
$$
commutates.  Hence 
$$
\pi((\iota\otimes \text{id})(\lambda))=\iota(\pi(\lambda))=\iota(\partial_vu)=0.
$$  
Using that the product maps for $C$ and $D$ are injective, this gives us that $(\iota\otimes \text{id})(\lambda)=0$.  

Now, we may write 
$$
\lambda=\sum_{i=1}^k \lambda_i\otimes \mu_i+\sum_{i=k+1}^m \lambda_i\otimes \mu_i
$$ 
for some $k\leq m$, where $\lambda_i\in K_0(C\cap D)$ for $i\leq k$, $\lambda_i\in K_0(S(C\cap D))$ for $i>k$, and similarly $\mu_i\in K_0(B)$ for $i\leq k$ and $\mu_i\in K_0(SB)$ for $i>k$.  As $K_*(B)$ is free, we may assume moreover that the set $\{\mu_1,...,\mu_m\}$ generates a free direct summand of $K_0(B)\oplus K_0(SB)$.  We then have that 
$$
(\iota\otimes \text{id})(\lambda)=\sum_{i=1}^m \iota(\lambda_i)\otimes \mu_i=0,
$$
which forces $\iota(\lambda_i)=0$ for each $i$ by assumption that the collection $\{\mu_1,...,\mu_m\}$ generates a free direct summand of $K_0(B)\oplus K_0(SB)$.   Applying Lemma \ref{iota lem} to each $\lambda_i$ separately gives us $l_1,...,l_m\in \N$ and invertible elements $w_1,...,w_m$ with 
$$
w_i\in \left\{\begin{array}{ll} M_{l_i}(\widetilde{A}) & i\leq k \\ M_{l_i}(\widetilde{SA}) & i> k \end{array}\right.
$$
and corresponding lifts $v_1,...,v_m$ with 
$$
v_i\in \left\{\begin{array}{ll} M_{2l_i}(\widetilde{A}) & i\leq k \\ M_{2l_i}(\widetilde{SA}) & i> k \end{array}\right.
$$
such that $\partial_{v_i}(w_i)=\lambda_i$ and $\partial_{v_i^{-1}}(w_i^{-1})=-\lambda_i$.   It will be important that there is $c>0$ such that for $i\leq k$, each $v_i$ is an $(\epsilon,c,C,D)$-lift of $u_i$ for \emph{any} $\epsilon>0$, and similarly for $i>k$, with $SC$ and $SD$ in place of $C$ and $D$.

Now, write $\mu_i=[p_i]-[q_i]$ for projections $p_i$ and $q_i$ in matrix algebras over $\widetilde{B}$ for $i\leq k$, and over $\widetilde{SB}$ for $i>k$.   Let $\{\alpha_t:S^2(A\otimes B)\leadsto A\otimes B\otimes \mathcal{K}\}$ be an asymptotic family inducing the inverse Bott map  as in Lemma \ref{bott lem}.  With notation as in Definition \ref{prod form}, let us define 
\begin{align*}
\mathbf{u}&:=u  \oplus (w_1^{-1}\boxtimes p_1)\oplus (w_1\boxtimes q_1)\oplus \cdots \oplus (w_k^{-1}\boxtimes p_k)\boxplus (w_k\boxtimes q_k) \\
&  \oplus\alpha_t(w_{k+1}^{-1}\boxtimes p_{k+1})\oplus \alpha_t(w_{k+1}\boxtimes q_{k+1})\oplus \cdots \oplus \alpha_t(w_m^{-1}\boxtimes p_m)\oplus \alpha_t(w_m\boxtimes q_m),
\end{align*}
and with notation also as in Lemma \ref{block lem} define  
\begin{align*}
\mathbf{v}:=v & \boxplus (v_1^{-1}\boxtimes p_1)\boxplus (v_1\boxtimes q_1)\boxplus \cdots \boxplus (v_k^{-1}\boxtimes p_k)\boxplus (v_k\boxtimes q_k) \\
&  \alpha_t(v_{k+1}^{-1}\boxtimes p_{k+1})\boxplus \alpha_t(v_{k+1}\boxtimes q_{k+1})\boxplus \cdots \boxplus \alpha_t(v_m^{-1}\boxtimes p_m)\boxplus \alpha_t(v_m\boxtimes q_m)
\end{align*}
which we can think of as elements of $M_{n_t}(\widetilde{A}\otimes \widetilde{B})$ and $M_{n_t}(\widetilde{A}\otimes \widetilde{B})$ respectively for some $n_t\in \N$ depending on $t$ (recall from Lemma \ref{bott lem} that each $\alpha_t:S^2(A\otimes B)\to A\otimes B\otimes\mathcal{K}$ takes image in $M_{m_t}(A\otimes B)$ for some $m_t\in \N$ depending on $t$).  Then Lemma \ref{block lem} gives that as long as our original $\delta$ was sufficiently small, we have
\begin{align*}
\partial_{\mathbf{v}}(\mathbf{u}) & =\partial_{v}(u)+\sum_{i=1}^k \partial_{v_i^{-1}\boxtimes p_i}(w_i^{-1}\boxtimes p_i)+\sum_{i=1}^k \partial_{v_i\boxtimes q_i}(w_i\boxtimes q_i) \\ 
& +\sum_{i=k+1}^m \partial_{\alpha_t(v_i^{-1}\boxtimes p_i)}\alpha_t(w_i^{-1}\boxtimes p_i)+\sum_{i=k+1}^m \partial_{\alpha_t(v_i\boxtimes q_i)}\alpha_t(w_i\boxtimes q_i).
\end{align*}
On the other hand, Lemmas \ref{bound com}, \ref{bott bound}, and \ref{bott lem} give that for suitably large $t$ this equals
\begin{align*}
\partial_{v}(u) & +\sum_{i=1}^k \partial_{v_i^{-1}\boxtimes p_i}(w_i^{-1}\boxtimes p_i)+\sum_{i=1}^k \partial_{v_i\boxtimes q_i}(w_i\boxtimes q_i) \\ 
& +\sum_{i=k+1}^m \alpha_*(\partial_{v_i^{-1}\boxtimes p_i}(w_i^{-1}\boxtimes p_i))+\sum_{i=k+1}^m \alpha_*(\partial_{v_i\boxtimes q_i}(w_i\boxtimes q_i)) \\ 
= \partial_{v}(u) & +\sum_{i=1}^k \partial_{v_i^{-1}}(w_i^{-1})\times [p_i]+\sum_{i=1}^k \partial_{v_i}(w_i)\times [q_i] \\ 
& +\sum_{i=k+1}^m \alpha_*(\partial_{v_i^{-1}}(w_i^{-1})\times [p_i])+\sum_{i=k+1}^m \alpha_*(\partial_{v_i}(w_i)\times [q_i]) \\ 
 = \partial_{v}(u) & + \sum_{i=1}^k (-\lambda_i)\times [p_i]+\sum_{i=1}^k \lambda_i\times [q_i] \\ 
&+\sum_{i=k+1}^m \bt((-\lambda_i\times [p_i])+\sum_{i=k+1}^m \bt(\lambda_i\times [q_i]) \\
= \partial_v(u) & -\sum_{i=1}^k \lambda_i\times \mu_i-\sum_{i=k+1}^m \bt(\lambda_i\times \mu_i) \\
=\partial_v(u) & -\pi(\lambda),
\end{align*}
and this last line is zero.  

We have just shown that $\partial_{\mathbf{v}}(\mathbf{u})=0$.  Noting that $[\mathbf{u}]$ defines a class in $K_1(A\otimes \widetilde{B})$ by Lemma \ref{prod nu}, it follows at this point from Proposition \ref{bound lem} that (as long as the original $\delta>0$ was suitably small) there exists $\nu\in K_1(C\otimes \widetilde{B})\oplus K_1(D\otimes \widetilde{B})$ such that $\sigma(\nu)=[\mathbf{u}]$.  Moreover, if we define 
$$
\xi:=\sum_{i=1}^m [w_i]\otimes [p_i]+\sum_{i=1}^m [w_i^{-1}]\otimes [q_i]= \sum_{i=1}^m [w_i]\otimes \mu_i\in K(A)\otimes_1 K(B)
$$
then  we have by definition of $\mathbf{u}$ that
$$
\sigma(\nu)=[\mathbf{u}]=[u]-\pi(\xi).
$$
Using surjectivity of the product maps for $C$ and $D$, and with notation as in Definition \ref{prod part}, we may lift $\nu$ to some $\zeta\in K(C)\otimes_1 K(B)\oplus K(D)\otimes_1 K(B)$.  Lemma \ref{iota com} gives commutativity of the diagram 
$$
\xymatrix{ K(C)\otimes_1 K(B)\oplus K(D)\otimes_1 K(B) \ar[r]^-{\sigma\otimes \text{id}} \ar[d]^-\pi & K(A)\otimes_1 K(B) \ar[d]^-\pi \\ K_1(C\otimes B)\oplus K_1(D\otimes B) \ar[r]^-\sigma & K_1(A\otimes B) },
$$
which implies that 
$$
[u]=\pi(\xi)+\sigma(\nu)=\pi(\xi)+\sigma(\pi(\zeta))=\pi(\xi)+\pi((\sigma\otimes \text{id})(\zeta))=\pi(\xi+(\sigma\otimes \text{id})(\zeta)), 
$$
so we have that $[u]$ is in the image of the map $\pi$, and are done.
\end{proof}

\section{Injectivity of the product map}\label{inj sec}

Finally, in this section we complete the main part of the paper by proving the injectivity half of Theorem \ref{main}.

\begin{theorem}\label{kun inj}
Let $A$ be a $C^*$-algebra, and say $A$ admits a uniform approximate ideal structure over a class $\mathcal{C}$ such that for each $(C,D)\in \mathcal{C}$, $C$, $D$, and $C\cap D$ satisfy the K\"{u}nneth formula.  Then for any $C^*$-algebra $B$ with free abelian $K$-theory, the product map 
$$
\times:K_*(A)\otimes K_*(B)\to K_*(A\otimes B)
$$
is injective.
\end{theorem}

\begin{proof}
With notation as in Definition \ref{prod part}, it suffices to show that the maps 
$$
\pi:K(A)\otimes_0 K(B)\to K_0(A\otimes B) \quad \text{and}\quad \pi:K(A)\otimes_1 K(B)\to K_1(A\otimes B)
$$
defined there are injective for any $B$ with $K_*(B)$ free abelian.  On replacing $B$ with its suspension, it suffices just to show injectivity in the $K_1$ case.  

Consider then an element $\kappa\in K(A)\otimes_1 K(B)$ such that $\pi(\kappa)=0$.  We will show that $\kappa=0$.  Fix a very small $\delta>0$, to be determined by the rest of the proof.

We may assume $\kappa$ is of the form
$$
\kappa=\sum_{i=1}^k \kappa_i\otimes ([p_i]-[q_i])+\sum_{i=k+1}^m \kappa_i\otimes ([p_i]-[q_i]) ,
$$
where for some $n\in \N$, each $\kappa_i$ is an element of $K_1(A)$ for $i\leq k$ or of $K_1(SA)$ for $i>k$, and each pair $p_i,q_i$ consists of projections in $M_n(\widetilde{B})$ for $i<k$ or in $M_n(\widetilde{SB})$ for $i>k$, so that the difference is in $M_n(B)$ or $M_n(SB)$ as appropriate, and so that the collection $([p_i]-[q_i])_{i=1}^n$ constitutes part of a basis for the free abelian group $K_0(B)\oplus K_0(SB)$.   Using Proposition \ref{v lem}, we may assume that for $i\leq k$ there is a finite-dimensional subspace $X_{0,i}$ of $A$ and invertible $u_i\in M_n(\widetilde{A})$ with the properties stated there for $\kappa_i$ and $\delta$; and similarly for each $i>k$, a finite-dimensional subspace $X_{1,i}$ of $SA$ and invertible $u_i\in M_n(\widetilde{SA})$ with the properties stated in Proposition \ref{v lem} with respect to $\kappa_i$ and $\delta$.

With notation `$\boxtimes$' as in Definition \ref{prod form}, 
and with 
$$
\{\alpha_t:S^2(\widetilde{A}\otimes \widetilde{B})\to \widetilde{A}\otimes \widetilde{B}\otimes \mathcal{K}\}
$$
an asymptotic family for $\widetilde{A}\otimes \widetilde{B}$ as in Lemma \ref{bott lem} that realizes the inverse Bott periodicity isomorphism, define
\begin{equation}\label{im alpha}
u_t:=\bigoplus_{i=1}^k u_i\boxtimes p_i\oplus \bigoplus_{i=1}^k u_i^{-1}\boxtimes q_i +\bigoplus_{i=k+1}^m \alpha_t(u_i\boxtimes p_i)\oplus \bigoplus_{i=k+1}^m \alpha_t(u_i^{-1}\boxtimes q_i).
\end{equation}
Then for all $t$ suitably large, $[u_t]$ defines a class in $K_1(\widetilde{A}\otimes \widetilde{B})$, which we may  consider as a class in $K_1(A\otimes B)$ thanks to Lemma \ref{prod nu}.  

By definition of $\pi$, there is $t_0\in [1,\infty)$ such that $\pi(\kappa)=[u_{t}]$ for all $t\geq t_0$, and so that the map
$$
[t_0,\infty)\to \bigcup_{n=1}^\infty M_n(A), \quad t\mapsto u_t
$$
is a continuous path of invertibles.  As $[u_{t_0}]=\pi(\kappa)=0$, we may assume moreover that there exist $l,p\in \N$ and a homotopy $\{w_s\}_{s\in [0,1]}$ of invertible elements in $M_p(\widetilde{A}\otimes \widetilde{B})$ such that $w_0\oplus 1_l =u_{t_0}$, such that  $w_1=1_p$, such that each $w_s$ and $w_s^{-1}$ are in $\{1+x\in M_p(\widetilde{A}\otimes \widetilde{B})\mid x\in M_p(A\otimes\widetilde{B})\}$.  Let $X_3$ be a finite-dimensional subspace of $A\otimes\widetilde{B}$ such that all $w_s$ and $w_s^{-1}$ are $\delta$-in $\{1+x\in M_p(\widetilde{A}\otimes \widetilde{B})\mid x\in M_p(X_3)\}$.  Using part \eqref{bl fd cont} of Lemma \ref{bott lem}, there is moreover a finite-dimensional subspace $X_4$ of $A\otimes\widetilde{B}$ such that for all $t\geq t_0$ there exists $n_t\in \N$ such that $u_t$ and $u_t^{-1}$ are $\delta$-in $M_{n_t}(X_4)$.   

Now, using Corollary \ref{tp cor} and Lemma \ref{tens lem}, for any $\delta>0$ there exists a triple $(h,C,D)$ such that $(1\otimes h,SC,SD)$ is an $f$-uniform $\delta$-ideal structure of $X_1$, and such that $(h\otimes 1,C\otimes \widetilde{B},D\otimes \widetilde{B})$ is an $f$-uniform $\delta$-ideal structure of both $X_3$ and $X_4$.    If $\delta$ is small enough, Proposition \ref{v lem} and Lemma \ref{bound lem inv} then let us build for each $i$ an invertible element 
$$
v_i\in \left\{\begin{array}{ll} M_{2n}(\widetilde{A}) & 1\leq i\leq k \\M_{2n}(\widetilde{SA}) & k+1\leq i\leq m\end{array}\right.
$$ 
such that for some $c>0$ and $\delta_1$ that tends to zero as $\delta$ tends to zero, we have that $v_i$ and $v_i^{-1}$ is a $(\delta_1,c,C,D)$ lift of $u_i$ and $u_i^{-1}$ respectively for $1\leq i\leq k$, and so that $v_i$ and $v_i^{-1}$ are a $(\delta_1,c,SC,SD)$ lift for $u_i$ and $u_i^{-1}$ respectively for $k+1\leq i\leq m$.  With notation as in Lemma \ref{block lem}, define also 
$$  
v:=\bigboxplus_{i=1}^k (v_i\boxtimes p_i)\boxplus \bigboxplus_{i=k}^n (v_i^{-1}\boxtimes q_i)\quad \text{and}\quad v_S:=\bigboxplus_{i=k+1}^m (v_i\boxtimes p_i)\boxplus \bigboxplus_{i=k+1}^m (v_i^{-1}\boxtimes q_i),
$$
which are elements of matrix algebras over $\widetilde{A}\otimes \widetilde{B}$ and $\widetilde{SA}\otimes \widetilde{SB}$ respectively.  Define also
$$
u:=\bigoplus_{i=1}^k u_i\boxtimes p_i\oplus \bigoplus_{i=1}^k u_i^{-1}\boxtimes q_i 
$$
and 
$$
u_S:=\bigoplus_{i=k+1}^m u_i\boxtimes p_i\oplus \bigoplus_{i=k+1}^m u_i^{-1}\boxtimes q_i.
$$
Then as long as $\delta>0$ is sufficiently small, Lemmas \ref{block lem} and \ref{bound com} give boundary classes $\partial_v u\in K_0((C\cap D)\otimes \widetilde{B})$ and $\partial_{v_S}(u_S)\in K_0(S(C\cap D) \otimes \widetilde{SB})$.  

Now, with $\beta^{-1}$ the inverse Bott periodicity map, the element 
$$
(\text{id}\oplus \beta^{-1})(\partial_vu,\partial_{v_S}(u_S))\in K_0((C\cap D)\otimes \widetilde{B})
$$
is necessarily zero.  Indeed, using Lemmas \ref{bott bound} and \ref{block lem}, this element is represented by
$$
\partial_vu+\partial_{\alpha_t{v_S}}(\alpha_t(u_S)))=\partial_{v\boxplus \alpha_t(v_S)}(u\oplus \alpha_t(u_S))
$$
for suitably large $t$.  With notation as in line \eqref{im alpha}, this equals $\partial_{v\boxplus \alpha_t(v_S)}(u_t)$.  Now, we can drag a homotopy between $u_t$ and $1$ through the construction of Proposition \ref{v lem} to produce a homotopy between this element and $1$ (this uses our choice of $(h,C,D)$, and the fact that there is a homotopy through invertibles between $u_t$ and $1$ that is close to $(1_{n_t}+M_{n_t}(X_3))\cup (1_{n_t}+M_{n_t}(X_4))$ for some appropriate $n_t\in \N$).

Lemmas \ref{bound com} and \ref{block lem} give then that 
$$
\pi\Bigg(\sum_{i=1}^m \partial_{v_i} (u_i)\otimes ([p_i]-[q_i])\Bigg)=(\text{id}\oplus \beta^{-1})(\partial_vu,\partial_{v_S}(u_S))
$$
whence the class
$$
\pi\Bigg(\sum_{i=1}^m \partial_{v_i} (u_i)\otimes ([p_i]-[q_i])\Bigg)\in K_0((C\cap D)\otimes \widetilde{B})
$$
is zero also.   Hence by injectivity of the product map for $C\cap D$, we have that 
$$
\sum_{i=1}^m \partial_{v_i} (u_i)\otimes ([p_i]-[q_i])
$$
is zero in $K(C\cap D)\otimes_0 K(B)$. Using the assumption that the collection $([p_i]-[q_i])_{i=1}^m$ forms part of a basis for $K_0(B)\oplus K_0(SB)$, we get that $\partial_{v_i}(u_i)=0$ in $K_0(C\cap D)\oplus K_0(S(C\cap D))$ for each $i$.  Hence Proposition \ref{bound lem} gives us $j,l\in \N$ and invertible elements 
$$
s_i\in \left\{\begin{array}{ll} M_{j+l}(\widetilde{D}) & 1\leq i\leq k \\ M_{j+l}(\widetilde{SD}) & k+1\leq i\leq m \end{array}\right.
$$
such that for each $i\in \{1,...,k\}$ we have that $(u_i\oplus 1_l)s_i^{-1}$ is in $M_{j+l}(\widetilde{C})$, and such that for each $i\in \{k+1,...,m\}$ we have that $(u_i\oplus 1_l)s_i^{-1}$ is in $M_{j+l}(\widetilde{SC})$.  Applying the same reasoning with the roles of $u_i$ and $u_i^{-1}$ interchanged, we similarly get invertible elements 
$$
r_i\in \left\{\begin{array}{ll} M_{j+l}(\widetilde{C}) & 1\leq i\leq k \\ M_{j+l}(\widetilde{SC}) & k+1\leq i\leq m \end{array}\right.
$$
such that for each $i\in \{1,...,k\}$, we have that $(u_i^{-1}\oplus 1_l)r_i^{-1}$ is in $M_{j+l}(\widetilde{C})$, and for each $i\in \{k+1,...,m\}$, we have that $(u_i^{-1}\oplus 1_l)r_i^{-1}$ is in $M_{j+l}(\widetilde{SC})$. 

Now, consider the class $\lambda\in \big(K(C)\otimes_1 K(B)\big)\oplus \big(K(D)\otimes_1 K(B)\big)$ defined by $\lambda=(\lambda_C,\lambda_D)$ where 
\begin{align*}
\lambda_C :=\sum_{i=1}^m [(u_i\oplus 1_l)s_i^{-1}]\otimes [p_i] + \sum_{i=1}^m [(u_i^{-1}\oplus 1_l)r_i^{-1}]\otimes [q_i]  
\end{align*}
and 
$$
\lambda_D:=\sum_{i=1}^m [s_i]\otimes [p_i] + \sum_{i=1}^m  [r_i]\otimes [q_i],
$$
and note that $\kappa=\sigma(\lambda)$.  The image of $\lambda$ under the product map 
\begin{align*}
\times:K(C) & \otimes_1 K(B) \oplus K(D)\otimes_1 K(B) \\ &\to \big(K_1(C\otimes B)\oplus K_1(SC\otimes SB)\big)\oplus \big(K_1(D\otimes B)\oplus K_1(SD\otimes SB)\big)
\end{align*}
is represented by the invertible element
\begin{align*}
x:=\Bigg(\bigoplus_{i=1}^m \big((u_i\oplus 1_l)s_i^{-1}\boxtimes p_i\big) \oplus \bigoplus_{i=1}^m  \big((u_i^{-1} & \oplus 1_l)r_i^{-1}\boxtimes q_i\big) ~,~\\ & \bigoplus_{i=1}^m \big(s_i\boxtimes p_i\big) \oplus  \bigoplus_{i=1}^m \big(r_i\boxtimes q_i\big)\Bigg). 
\end{align*}
We have that $\pi(\lambda)$ equals the image of the class above under the map 
\begin{align*}
\text{id}\oplus \beta^{-1}:\big(K_1(C\otimes B)\oplus K_1(SC\otimes SB)\big) & \oplus \big(K_1(D\otimes B)\oplus K_1(SD\otimes SB)\big) \\ & \to K_1(C\otimes B)\oplus K_1(D\otimes B),
\end{align*}
which, with notation as in Lemma \ref{bott lem}, is represented concretely by the invertible element $(\text{id}\oplus \alpha_t)(x)$ for all suitably large $t$.  On the other hand, using almost multiplicativity of the asymptotic family $\{\alpha_t\}$ and comparing this with the formula for $u_t$ in line \eqref{im alpha}, we see that $u_t$ can be made arbitrarily close to 
\begin{align*}
(\text{id}\oplus \alpha_t)\Bigg(\Big(\bigoplus_{i=1}^m \big((u_i\oplus 1_l)s_i^{-1}\boxtimes p_i\big) & \oplus \bigoplus_{i=1}^m  \big((u_i^{-1}\oplus 1_l)r_i^{-1}\boxtimes q_i\big) \Big) \\ &~\cdot ~\Big(\bigoplus_{i=1}^m \big(s_i\boxtimes p_i\big) \oplus  \bigoplus_{i=1}^m \big(r_i\boxtimes q_i\big)\Big)\Bigg)
\end{align*}
by increasing $t$, and up to taking block sum with $1_q$ for some $q$ depending on $t$.

Now, for each fixed $t$ there is $n_t\in \N$ such that $u_t$ is homotopic to the identity through invertibles that are $\delta$-in 
$$
\{1+x\in M_{n_t}(A\otimes \widetilde{B})\mid x\in M_{n_t}(X_3)\cup M_{n_t}(X_4)\}
$$
via the concatenation of the homotopies $\{u_s\}_{s\in [t_0,t]}$ and $\{w_s\oplus 1_{n_t-p}\}_{s\in [0,1]}$ and our assumption on $(h,C,D)$.  We are thus in a position to apply Proposition \ref{sigma lem} to conclude that there exists a class $\mu\in K_1((C\cap D)\otimes \widetilde{B})$ such that $\iota(\mu)=\pi(\lambda)$.  Using surjectivity of the product map for $C\cap D$, we may lift $\mu$ to some element $\nu$ of $K(C\cap D)\otimes_1 K(\widetilde{B})$.  Using Lemma \ref{iota com}, we have that 
$$
\pi(\lambda)=\iota(\mu)=\iota(\pi(\nu))=\pi(\iota(\nu)).
$$
Hence by injectivity of the product maps for $C$ and $D$, this forces $\lambda=\iota(\nu)$.  Finally, we have that $\kappa=\sigma(\lambda)$ and so
$$
\kappa=\sigma(\lambda)=\sigma(\iota(\nu)).
$$
However, $\sigma\circ \iota$ is clearly the zero map on $K$-theory, so we are done.
\end{proof}

\appendix

\section{Nuclear dimension}\label{nd1 app}

In this appendix, we give examples of weak approximate ideal structures coming from nuclear dimension one.  See \cite{Winter:2010eb} for background on the theory of nuclear dimension.

For the statement of the next result, if $A$ is a $C^*$-algebra, let $A_\infty$ denote the quotient $\prod_\N A / \oplus_\N A$ of the product of countably many copies of $A$ by the direct sum.  If $(B_n)$ is a sequence of $C^*$-suablgebras of $A$, we let $B_\infty$ denote the $C^*$-subalgebra $\prod_\N B_n / \oplus_\N B_n$ of $A_\infty$.

The following fact was told to me by Wilhelm Winter\footnote{Professor Winter probably knows a better proof!}.

\begin{proposition}\label{nd1}
Let $A$ be a separable\footnote{Not really necessary, but the statement would be a little fiddlier otherwise.} unital $C^*$-algebra of nuclear dimension one.  Then there exist
\begin{enumerate}[(i)]
\item a positive contraction $h\in A_\infty\cap A'$, and
\item sequences $(C_n)$, $(D_n)$ of $C^*$-subalgebras of $A$
\end{enumerate}
such that:
\begin{enumerate}
\item each $C_n$ and each $D_n$ is a quotient of a cone over a finite-dimensional $C^*$-algebra, 
\item for all $a\in A$, $ha\in C_\infty$, $(1-h)a\in D_\infty$,
\end{enumerate}
\end{proposition}

\begin{proof}
Using \cite[Theorem 3.2]{Winter:2010eb} (and that $A$ is separable) there exists a sequence $(\psi_n,\phi_n,F_n)$ where: 
\begin{enumerate}[(i)]
\item each $F_n$ is a finite-dimensional $C^*$-algebra that decomposes as a direct sum $F_n=F_n^{(0)}\oplus F_n^{(1)}$;
\item each $\psi_n$ is a ccp map $A\to F_n$ such that the induced diagonal map 
$$
\overline{\psi}:A\to F_\infty
$$ 
is order zero;
\item each $\phi_n$ is a map $F_n\to A$ such that the restriction $\phi_n^{(i)}$ of $\phi_n$ to $F_n^{(i)}$ is ccp and order zero;
\item for each $a\in A$, $\phi_n\psi_n(a)\to a$ as $n\to\infty$.
\end{enumerate}
Let $\overline{\phi}:F_\infty\to A_\infty$, and $\overline{\phi^{(i)}}:F^{(i)}_\infty\to A_\infty$ denote the induced maps, let $\kappa^{(i)}:F_\infty\to F_\infty^{(i)}$ denote the canonical quotient, and consider the composition 
$$
\theta^{(i)}:=\overline{\phi^{(i)}}\circ \kappa^{(i)}\circ \overline{\psi}:A\to A_\infty.
$$
Each $\theta^{(i)}$ is then ccp and order zero, and we have moreover that $\theta^{(0)}+\theta^{(1)}:A\to A_\infty$ agrees with the canonical diagonal inclusion.

Now, let $M_i:=M(C^*(\theta^{(i)}(A)))$ be the multiplier algebra of the $C^*$-subalgebra $C^*(\theta^{(i)}(A))$ of $A_\infty$ generated by $\theta^{(i)}(A)$.  Using \cite[Theorem 2.3]{Winter:1009aa} if we set $h_i:=\theta^{(i)}(1)$, then $h_i$ is a positive contraction in $C^*(\theta^{(i)}(A))\cap A'$, and there exists a unital\footnote{Unitality follows from the proof in the given reference, but does not appear explicitly in the statement.} $*$-homomorphism $\pi^{(i)}:A\to M_i\cap \{h_i\}'$ such that 
$$
\theta^{(i)}(a)=h_i\pi^{(i)}(a)
$$
for all $a\in A$.  As $1=\theta^{(0)}(1)+\theta^{(1)}(1)=h_1+h_2$, we will switch notation and write $h:=h_1$, so $1-h=h_2$, so for all $a\in A$,
\begin{equation}\label{split}
a=h\pi^{(0)}(a)+(1-h)\pi^{(1)}(a).
\end{equation}
Note in particular that $h$ commutes with both $\theta^{(1)}(A)$ (as $h=h_1$ and $h_1$ commutes with this collection), and with $\theta^{(2)}(A)$ (as $1-h=h_2$, and $h_2$ commutes with this collection). Hence $h$ commutes with $\theta^{(1)}(A)+\theta^{(2)}(A)\supseteq A$, so in particular $h$ is in $A_\infty\cap A'$.  

Now, let us think of $\pi^{(i)}:A\to M_i$ as having image in the double dual $(A_\infty)^{**}$ by postcomposing with the canonical embedding $M_i\to (A_\infty)^{**}$.  Let us replace $\pi^{(i)}$ with the map 
\begin{equation}\label{cut down}
a\mapsto \chi_{[0,1]\setminus \{i\}}(h)\pi^{(i)}(a)+\chi_{\{i\}}(h)a.
\end{equation} 
Then the equation in line \eqref{split} still holds for all $a\in A$.  Let $B$ be the unital $C^*$-algebra generated by $h$, $A$, $\pi^{(0)}(A)$ and $\pi^{(1)}(A)$, and note that $h$ is central in $B$.  For each $\lambda\in [0,1]$ in the spectrum of $h$ in $C^*(h,1)$, let $I_{\lambda}$ be the $C^*$-ideal in $B$ generated by the corresponding maximal ideal in $C^*(h,1)$ (with $I_\lambda=B$ if $\lambda$ is not in the spectrum of $h$).  Then in $B/I_\lambda$, the equation in line \eqref{split} descends to
$$
a=\lambda\pi^{(0)}(a)+(1-\lambda)\pi^{(1)}(a).
$$
If $\lambda\in (0,1)$ and $a=u\in A$ is unitary, this writes the image of $u$ in $B/I_\lambda$ as a convex combination of two elements in the unit ball; as unitaries are always extreme points in the unit ball of a $C^*$-algebra \cite[Theorem II.3.2.17]{Blackadar:2006eq}, this is impossible unless $\pi^{(0)}(u)=\pi^{(1)}(u)=u$ modulo $I_\lambda$ for all $\lambda\in (0,1)$.  As the unitaries span any unital $C^*$-algebra \cite[Proposition II.3.2.12]{Blackadar:2006eq}, this forces $\pi^{(0)}(a)=\pi^{(1)}(a)=a$ modulo $I_\lambda$ for all $a\in A$ and all $\lambda\in (0,1)$.  On the other hand, if $\lambda=0$, we clearly get $\pi^{(1)}(a)=a$ modulo $I_0$ for all $a\in A$, while $\pi^{(0)}(a)=a$ modulo $I_0$ follows from the replacement we made in line \eqref{cut down}.  Similarly, if $\lambda=0$, we also get that $\pi^{(0)}(a)=a$ and $\pi^{(1)}(a)=a$ modulo $I_1$.  Putting this together, we have that the postcomposition of either $\pi^{(0)}$ or $\pi^{(1)}$ with the natural diagonal $*$-homomorphism
$$
\Phi:B\to \prod_{\lambda\in \text{spectrum}(h)}B/I_\lambda
$$
agrees with the natural map $A\to  \prod_{\lambda\in [0,1]}B/I_\lambda$ induced by the inclusion $A\to B$.  However, as $C^*(h,1)$ is contained in the center of $B$, the map $\Phi$ is injective by \cite[Theorem 7.4.2]{Douglas:1972vn}.  Hence we get that both $\pi^{(0)}$ and $\pi^{(1)}$ agree with the diagonal inclusion $A\to A_\infty$, and thus have the equations
$$
\theta^{(0)}(a)=ha \quad\text{and}\quad \theta^{(1)}(a)=(1-h)a
$$
for all $a\in A$.

To complete the proof, therefore, we need to find sequences $(C_n)$ and $(D_n)$ of $C^*$-subalgebras of $A$ with the right properties.  For each $n$ and each $i\in \{0,1\}$, consider $\phi_n^{(i)}:F_n^{(i)}\to A$.  As this is order zero, \cite[Corollary 3.1]{Winter:1009aa} gives a $*$-homomorphism $\rho^{(i)}_n:C_0(0,1]\otimes F_n^{(i)}\to A$ such that $\phi_n^{(i)}(b)=\rho_n^{(i)}(x\otimes b)$ for all $b\in A$, where $x\in C_0(0,1]$ is the identity function.  Set $C_n:=\rho_n^{(0)}(C_0(0,1]\otimes F_n^{(0)})$ and $D_n=\rho_n^{(1)}(C_0(0,1]\otimes F_n^{(0)})$, which contain the images of $\phi_n^{(0)}$ and $\phi_n^{(1)}$ respectively.  It is straightforward to check that $(C_n)$ and $(D_n)$ have the right properties.
\end{proof}

The next corollary follows by lifting the element $h\in A_\infty$ to a positive contraction $(h_n)\in \prod_n A$: we leave the details to the reader.

\begin{corollary}\label{approx cor}
Let $A$ be a separable $C^*$-algebra of nuclear dimension one, and let $\mathcal{C}$ be the class of pairs $(C,D)$ of $C^*$-subalgebras of $A$ such that each of $C$ and $D$ is isomorphic to a quotient of a cone over a finite dimensional $C^*$-algebra.  Then $A$ has a weak approximate ideal structure over $\mathcal{C}$. \qed
\end{corollary}

\begin{remark}
Based on the above it is natural to ask: if $A$ admits a weak approximate ideal structure over a class $\mathcal{C}$ as in Definition \ref{weak ais}, can one use an additional argument to show that $A$ admits an approximate ideal structure over $\mathcal{C}$? We do not believe this is true due to the following example\footnote{Inspired by a suggestion of Ian Putnam.}; we warn the reader that we did not check the details of what follows.  It seems by adapting Proposition \ref{nd1} that one can show that if $A$ has nuclear dimension one and real rank zero, then it has a weak approximate ideal structure over the class $\mathcal{C}$ of pairs of its finite dimensional $C^*$-subalgebras.  In particular, this would apply to any Kirchberg algebra (see \cite[Theorem G]{Bosa:2014zr} and \cite[Proposition 4.1.1]{Rordam:2002cs}).  However, if $A$ admits an approximate ideal structure over a class of pairs of finite-dimensional $C^*$-algebras, then a mild elaboration of Proposition \ref{v lem} below shows that $K_1(A)$ is torsion free.  As there are Kirchberg algebras with non-trivial torsion $K_1$ group (see \cite[Section 4.3]{Rordam:2002cs}), this (if correct!) would show that admitting a weak approximate ideal structure over $\mathcal{C}$ and admitting an approximate ideal structure over $\mathcal{C}$ are not the same.
\end{remark}

\section{Finite dynamical complexity}\label{fdc app}

In this appendix, we give examples of excisive decompositions coming from decompositions of groupoids as introduced in \cite{Guentner:2014bh}.  Our conventions on groupoids will be as in \cite[Appendix A]{Guentner:2014bh} and \cite[Section 2.3]{Renault:2009zr}.

The following is a slight variant of \cite[Definition A.4]{Guentner:2014bh}.

\begin{definition}\label{gpd decomp def}
Let $G$ be a locally compact, Hausdorff, \'{e}tale groupoid, let $H$ be an open subgroupoid of $G$, and let $\mathcal{C}$ be a set of open subgroupoids of $G$.  We say that $H$ is \emph{decomposable} over $\mathcal{C}$ if for any open, relatively compact subset $K$ of $H$ there exists an open cover $H^{(0)}=U_0\cup U_1$ of the unit space of $H$ such that for each $i\in \{0,1\}$ the subgroupoid of $H$ generated by 
$$
\{h\in K\mid s(h)\in U_i\}
$$
is contained in an element of $\mathcal{C}$. 
\end{definition}

The first technical result of this section is as follows.  See Definitions \ref{intro decomp} and \ref{intro exc} for terminology.

\begin{proposition}\label{gpd exc decomp}
Say $G$ is a second countable, locally compact, Hausdorff \'{e}tale groupoid that that decomposes over a set $\mathcal{D}$ of open subgroupoids of $G$.  Then the reduced groupoid $C^*$-algebra $C^*_r(G)$ admits an approximate ideal structure over the class of pairs
$$
\mathcal{C}:=\{(C^*_r(H_1),C^*_r(H_2))\mid H_1,H_2\in \mathcal{D}\}.
$$
Moreover, if every groupoid in $\mathcal{D}$ is clopen, then $C^*_r(G)$ admits a uniform approximate ideal structure over the class $\mathcal{C}$ above.
\end{proposition}

The proof will proceed via some lemmas.  First we give the existence of approximate ideal structures.

\begin{lemma}\label{gpd decomp}
Say $G$ is a locally compact, Hausdorff \'{e}tale groupoid that decomposes over a set $\mathcal{D}$ of subgroupoids of $G$ in the sense of Definition \ref{gpd decomp def}.  Then the reduced groupoid $C^*$-algebra $C^*_r(G)$ admits an approximate ideal structure over the set $\mathcal{C}:=\{(C^*_r(H_1),C^*_r(H_2))\mid H_1,H_2\in \mathcal{D}\}$.
\end{lemma}

\begin{proof}
Let $X$ be a finite-dimensional subspace of $C^*_r(G)$.  Up to an approximation, we may assume that there is an open relatively compact subset $K$ of $G$ such that every element of $X$ is an element of $C_c(G)$ supported in $K$.  Using (a slight variation on) \cite[Lemma A.12]{Guentner:2014bh}, for any $\epsilon>0$, there is an open cover $G^{(0)}=U_0\cup U_1$ of the base space of $G$ and a pair of continuous compactly supported functions $\{\phi_0,\phi_1:G^{(0)}\to [0,1]\}$ with the following properties.  
\begin{enumerate}[(i)]
\item \label{supp phi} each $\phi_i$ is supported in $U_i$;
\item \label{small gen} for each $i\in \{0,1\}$, the set $\{k\in K\mid r(k)\in U_i\}$ generates an open subgroupoid of $G$ that is contained in some element $H_i$ of $\mathcal{D}$;
\item \label{pou} for each $x\in G^{(0)}$, $\phi_0(x)+\phi_1(x)=1$ and for each $k\in K$, $\phi_0(r(k))+\phi_1(r(k))=1$;
\item \label{low var} for any $k\in K$ and $i\in \{0,1\}$, $|\phi_i(s(k))-\phi_i(r(k))|<\epsilon$.
\end{enumerate}
We claim that for any $\delta>0$, there exists $\epsilon$ suitably small such that if $\phi_0$ and $\phi_1$ are chosen as above, then $(h,C,D)=(\phi_0,C^*_r(H_1),C^*_r(H_2))$ is a $\delta$-approximate ideal structure.

Indeed, the fact that $\|[h,a]\leq \delta\|a\|$ for all $a\in X$ follows from condition \eqref{low var} above and \cite[Lemma 8.20]{Guentner:2014aa}.  We have moreover that for any $a\in X$, $ha=\phi_0a$, and this is supported in $\{k\in K\mid r(k)\in U_0\}$ by condition \eqref{supp phi}, whence is in $C^*_r(H_0)$ by condition \eqref{small gen}.  On the other hand, $(1-h)a=\phi_1a$ for any $a\in X$ by condition \eqref{pou}, whence $(1-h)a$ is in $C^*_r(H_1)$ by the same argument.  
\end{proof}

The next lemma is presumably well-known.

\begin{lemma}\label{cond exp}
Let $G$ be a locally compact, Hausdorff, \'{e}tale groupoid, and let $H\subseteq G$ be a clopen subgroupoid.  Then the restriction map $E:C_c(G)\to C_c(H)$ extends to a conditional expectation $E:C^*_r(G)\to C^*_r(H)$.
\end{lemma}

\begin{proof}
For $x\in H^{(0)}$, let $\pi_x:C^*_r(H)\to \mathcal{B}(\ell^2(H_x))$ be the associated regular representation defined by 
$$
(\pi_x(b)\xi)(h):=\sum_{k\in H_x} b(hk^{-1})\xi(h)
$$
as in \cite[Section 2.3.4]{Renault:2009zr}.  Let $\xi,\eta\in \ell^2(H_x)$, and consider 
$$
\langle \xi,\pi_x(E(a))\eta\rangle_{\ell^2(H_x)}=\sum_{h,k\in H_x}E(a)(hk^{-1})\eta(k)\overline{\xi(h)}=\sum_{h,k\in G_x}a(hk^{-1})\tilde{\eta}(k)\overline{\tilde{\xi}(h)}
$$
where $\tilde{\xi}\in \ell^2(G_x)$ is the function defined by extending $\xi$ by zero on $G_x\setminus H_x$, and the second equality uses that $H$ is a subgroupoid to deduce that if $h,k\in H$, then $hk^{-1}$ is in $H$.  Hence if $\pi_x^G$ is the corresponding representation of $G$ on $\ell^2(G_x)$, we have 
$$
\langle \xi,\pi_x(E(a))\eta\rangle_{\ell^2(H_x)}=\langle \tilde{\xi},\pi_x^G(a)\tilde{\eta}\rangle,
$$
and so 
$$
\|E(a)\|=\sup_{\|\xi\|=\|\eta\|=1}|\langle \xi,\pi_x(E(a))\eta\rangle_{\ell^2(H_x)}|=\sup_{\|\xi\|=\|\eta\|=1}|\langle \tilde{\xi},\pi_x^G(a)\tilde{\eta}\rangle|\leq \|a\|.
$$
Hence $E$ is contractive, and so in particular extends to an idempotent linear contraction $E:C^*_r(G)\to C^*_r(H)$.  This extended map is necessarily a contraction by a classical theorem of Tomiyama: see for example \cite[Theorem 1.5.10]{Brown:2008qy}.
\end{proof}

\begin{lemma}\label{gpd exc}
Say $G$ is a locally compact, Hausdorff, \'{e}tale groupoid.  Then the set of pairs of $C^*$-subalgebras of $C^*_r(G)$ of the form $(C^*_r(H_1),C^*_r(H_2))$ with $H_1,H_2\subseteq G$ both open subgroupoids, and at least one of them also closed, is strongly excisive as in Definition \ref{intro exc}.
\end{lemma}

\begin{proof}
Say $B$ is an arbitrary $C^*$-algebra, and consider $c\in C^*_r(H_1)\otimes B$ and $d\in C^*_r(H_2)\otimes B$.  Say without loss of generality that $H_2$ is closed, and let $E:C^*_r(G)\to C^*_r(H_2)$ be the conditional expectation of Lemma \ref{cond exp}.  As $E$ is just defined on $C_c(G)$ by restriction of functions, it follows that $E$ takes $C^*_r(H_1)$ into itself, and therefore into $C^*_r(H_1)\cap C^*_r(H_2)$.  Hence by functoriality of tensor product maps, we see that $E\otimes \text{id}$ restricted to $C^*_r(H_1)\otimes B$ is a map
$$
E\otimes \text{id}:C^*_r(H_1)\otimes B\to (C^*_r(H_1)\cap C^*_r(H_2))\otimes B
$$
and in particular $(E\otimes \text{id})(d)$ is in $(C^*_r(H_1)\cap C^*_r(H_2))\otimes B$.  On the other hand, as $E\otimes \text{id}$ is contractive (see for example \cite[Theorem 3.5.3]{Brown:2008qy}) and takes $C^*_r(H_1)$ to itself, so we get that 
$$
\|c-(E\otimes \text{id})(d)\|=\|(E\otimes \text{id})(c-d)\|\leq \|c-d\|
$$
and 
$$
\|d-(E\otimes \text{id})(d)\|\leq \|c-d\|+\|c-(E\otimes \text{id})(d)\|\leq 2\|c-d\|
$$
so we are done.
\end{proof}

Proposition \ref{gpd exc decomp} now follows directly from Lemmas \ref{gpd decomp}, \ref{cond exp}, and \ref{gpd exc}.

We spend the rest of this appendix deriving some consequences of Proposition \ref{gpd exc decomp}.

\begin{corollary}\label{gpd decom cor}
Say $G$ is an ample second countable, locally compact, Hausdorff \'{e}tale groupoid.  Let $\mathcal{K}$ be the class of clopen subgroupoids of $G$, such that for any $H\in \mathcal{K}$, and any clopen subgroupoid $K$ of $\mathcal{H}$, $C^*_r(K)$ satisfies the K\"{u}nneth formula.  Then $\mathcal{K}$ is closed under decomposability.
\end{corollary}

\begin{proof}
Say $H$ is a clopen subgroupoid of $G$ that decomposes over $\mathcal{K}$.  Then $C^*_r(H)$ strongly excisively decomposes over the class $\{(C^*_r(K_1),C^*_r(K_2))\mid K_1,K_2\in \mathcal{K}\}$ by Proposition \ref{gpd exc decomp}, and so $C^*_r(H)$ satisfies K\"{u}nneth by Theorem \ref{main}.  The same argument also applies to any clopen subgroupoid of $H$: indeed, any clopen subgroupoid of $H$ is easily seen to also decompose over $\mathcal{K}$ (compare the proof of \cite[Lemma 3.16]{Guentner:2014bh}).
\end{proof}

We will finish with an example that is closely related to the notion of finite dynamical complexity for groupoids introduced in \cite[Definition A.4]{Guentner:2014bh}

\begin{definition}\label{fdc clopen}
Say $G$ is an ample, locally compact, Hausdorff \'{e}tale groupoid with finite dynamical complexity.  Let $\mathcal{C}$ be the class of compact open subgroupoids of $G$, and let $\mathcal{D}$ be the smallest class of clopen subgroupoids of $G$ containing $\mathcal{C}$ and closed under decomposability.   Then $G$ has \emph{strong finite dynamical complexity} if $G$ itself is contained in $\mathcal{D}$.
\end{definition}

\begin{theorem}\label{fdc cor}
Say $G$ is a principal, locally compact, Hausdorff \'{e}tale groupoid with strong finite dynamical complexity.  Then $C^*_r(G)$ satisfies the K\"{u}nneth formula.
\end{theorem}

This result is not new: groupoids as in the statement are amenable by \cite[Theorem A.9]{Guentner:2014bh}, and therefore their $C^*$-algebras satisfy the UCT by a result of Tu \cite[Proposition 10.7]{Tu:1999bq} (at least in the second countable case).  Nonetheless, it seems interesting to give a relatively direct proof based on the internal structure of the $C^*$-algebra.

\begin{proof}
Let $\mathcal{K}$ be as in Corollary \ref{gpd decom cor}, and let $\mathcal{C}$ be the class of compact open subgroupoids of $\mathcal{C}$.  Then for any $H\in \mathcal{K}$, the reduced $C^*$-algebra $C^*_r(H)$ is principal and proper, so Morita equivalent to the continuous functions $C(H^{(0)}/H)$ on the orbit space by \cite[Example 2.5 and Theorem 2.8]{Muhly:1987fk} (the second countability assumptions in that paper are not necessary in the \'{e}tale case \cite{Felix:2014aa}).  Hence $C^*_r(H)$ satisfies the K\"{u}nneth formula.  As $\mathcal{C}$ is closed under taking clopen subgroupoids, $\mathcal{K}$ contains $\mathcal{C}$.

Hence if $\mathcal{D}$ is as in Definition \ref{fdc clopen}, then $\mathcal{K}$ contains $\mathcal{D}$ by Corollary \ref{gpd decom cor}.  However, strong finite dynamical complexity implies that $G$ itself is in $\mathcal{D}$, so we are done.
\end{proof}

\begin{example}
Let $X$ be a bounded geometry metric space, and assume that $X$ has finite decomposition complexity as introduced in \cite{Guentner:2009tg} and studied in \cite{Guentner:2013aa}.  Then the associated coarse groupoid $G(X)$ has strong finite dynamical complexity by the proof of \cite[Theorem A.4]{Guentner:2014bh}.  Hence the associated groupoid $C^*$-algebra $C^*_r(G(X))$, which canonically identifies with the uniform Roe algebra $C^*_u(X)$, satisfies the K\"{u}nneth formula by Theorem \ref{fdc cor}.
\end{example}

\bibliography{Generalbib}

\end{document}